\documentclass[12pt, a4]{amsart}
\usepackage{latexsym,amsmath,amssymb,amsthm,mathrsfs,color}
\usepackage{float}
\usepackage{graphicx}
\usepackage{marginnote}
\usepackage{scalerel}
\usepackage{stackengine,wasysym}
\usepackage[bbgreekl]{mathbbol}
\usepackage{bbm}
\usepackage{tikz-cd}
\usetikzlibrary{graphs,decorations.pathmorphing,decorations.markings}
\usepackage[all]{xy}
\usepackage{stmaryrd}
\usepackage{chngcntr}
\counterwithin{figure}{section}
\usepackage{hyperref}
\hypersetup{colorlinks,%
	citecolor=brown,%
	filecolor=black,%
	linkcolor=blue,%
	urlcolor=black}
\usepackage[toc,page]{appendix}

\usepackage[shortalphabetic]{amsrefs}
\makeatletter
\def\append@label@year@{%
	\safe@set\@tempcnta\bib@year
	\edef\bib@citeyear{\the\@tempcnta}%
	\ifnum\bib@citeyear>9
	\append@to@stem{%
		\ifx\bib@year\@empty
		\else
		\@xp\year@short \bib@citeyear \@nil
		\fi
	}%
	\fi
}
\makeatother

\let\oldtocsection=\tocsection
\renewcommand{\tocsection}[2]{\hspace{0em}\oldtocsection{#1}{#2}}

\usepackage[a4paper]{geometry}

\def\upddots{\mathinner{\mkern 1mu\raise 1pt \hbox{.}\mkern 2mu
		\mkern 2mu \raise 4pt\hbox{.}\mkern 1mu \raise 7pt\vbox {\kern 7
			pt\hbox{.}}} }

\numberwithin{equation}{section}

\begin{document}
	\setlength{\unitlength}{2.5cm}
	\newtheorem{thm}{Theorem}[section]
	\newtheorem{lm}[thm]{Lemma}
	\newtheorem{prop}[thm]{Proposition}
	\newtheorem{cor}[thm]{Corollary}
	\newtheorem{conj}[thm]{Conjecture}
	\newtheorem{specu}[thm]{Speculation}
	
	\theoremstyle{definition}
	\newtheorem{dfn}[thm]{Definition}
	\newtheorem{eg}[thm]{Example}
	\newtheorem{rmk}[thm]{Remark}
	\newcommand{\ome}{\varpi}
	\newcommand{\F}{\mathbf{F}}
	\newcommand{\N}{\mathbbm{N}}
	\newcommand{\R}{\mathbbm{R}}
	\newcommand{\C}{\mathbbm{C}}
	\newcommand{\Z}{\mathbbm{Z}}
	\newcommand{\Q}{\mathbbm{Q}}
	\newcommand{\Mp}{{\rm Mp}}
	\newcommand{\Sp}{{\rm Sp}}
	\newcommand{\GSp}{{\rm GSp}}
	\newcommand{\GL}{{\rm GL}}
	\newcommand{\PGL}{{\rm PGL}}
	\newcommand{\SL}{{\rm SL}}
	\newcommand{\SO}{{\rm SO}}
	\newcommand{\Spin}{{\rm Spin}}
	\newcommand{\GSpin}{{\rm GSpin}}
	\newcommand{\Ind}{{\rm Ind}}
	\newcommand{\Res}{{\rm Res}}
	\newcommand{\Hom}{{\rm Hom}}
	\newcommand{\End}{{\rm End}}
	\newcommand{\msc}[1]{\mathscr{#1}}
	\newcommand{\mfr}[1]{\mathfrak{#1}}
	\newcommand{\mca}[1]{\mathcal{#1}}
	\newcommand{\mbf}[1]{{\bf #1}}
	\newcommand{\mbm}[1]{\mathbbm{#1}}
	\newcommand{\into}{\hookrightarrow}
	\newcommand{\onto}{\twoheadrightarrow}
	\newcommand{\s}{\mathbf{s}}
	\newcommand{\cc}{\mathbf{c}}
	\newcommand{\bfa}{\mathbf{a}}
	\newcommand{\id}{{\rm id}}
	\newcommand{\g}{ \mathbf{g} }
	\newcommand{\w}{\mathbbm{w}}
	\newcommand{\Ftn}{{\sf Ftn}}
	\newcommand{\p}{\mathbf{p}}
	\newcommand{\bq}{\mathbf{q}}
	\newcommand{\WD}{\text{WD}}
	\newcommand{\W}{\text{W}}
	\newcommand{\Wh}{{\rm Wh}}
	\newcommand{\Whc}{{{\rm Wh}_\psi}}
	\newcommand{\ggma}{\omega}
	\newcommand{\sct}{\text{\rm sc}}
	\newcommand{\Of}{\mca{O}^\digamma}
	\newcommand{\gk}{c_{\sf gk}}
	\newcommand{\Irr}{ {\rm Irr} }
	\newcommand{\Irrg}{ {\rm Irr}_{\rm gen} }
	\newcommand{\diag}{{\rm diag}}
	\newcommand{\uchi}{ \underline{\chi} }
	\newcommand{\Tr}{ {\rm Tr} }
	\newcommand{\der}\de
	\newcommand{\Stab}{{\rm Stab}}
	\newcommand{\Ker}{{\rm Ker}}
	\newcommand{\bfp}{\mathbf{p}}
	\newcommand{\bfq}{\mathbf{q}}
	\newcommand{\KP}{{\rm KP}}
	\newcommand{\Sav}{{\rm Sav}}
	\newcommand{\de}{{\rm der}}
	\newcommand{\tnu}{{\tilde{\nu}}}
	\newcommand{\lest}{\leqslant}
	\newcommand{\gest}{\geqslant}
	\newcommand{\tu}{\widetilde}
	\newcommand{\tchi}{\tilde{\chi}}
	\newcommand{\tomega}{\tilde{\omega}}
	\newcommand{\Rep}{{\rm Rep}}
	\newcommand{\cu}[1]{\textsc{\underline{#1}}}
	\newcommand{\set}[1]{\left\{#1\right\}}
	\newcommand{\ul}[1]{\underline{#1}}
	\newcommand{\ol}[1]{\overline{#1}}
	\newcommand{\wt}[1]{\widetilde{#1} }
	\newcommand{\wtsf}[1]{\wt{\sf #1}}
	\newcommand{\anga}[1]{{\left\langle #1 \right\rangle}}
	\newcommand{\angb}[2]{{\left\langle #1, #2 \right\rangle}}
	\newcommand{\wm}[1]{\wt{\mbf{#1}}}
	\newcommand{\elt}[1]{\pmb{\big[} #1\pmb{\big]} }
	\newcommand{\ceil}[1]{\left\lceil #1 \right\rceil}
	\newcommand{\floor}[1]{\left\lfloor #1 \right\rfloor}
	\newcommand{\val}[1]{\left| #1 \right|}
	\newcommand{\aff}{ {\rm aff} }
	\newcommand{\ex}{ {\rm ex} }
	\newcommand{\exc}{ {\rm exc} }
	\newcommand{\HH}{ \mca{H} }
	\newcommand{\HKP}{ {\rm HKP} }
	\newcommand{\std}{ {\rm std} }
	\newcommand{\motimes}{\text{\raisebox{0.25ex}{\scalebox{0.8}{$\bigotimes$}}}}
	
	\newcommand{\FF}{\mca{F}}
	\newcommand{\tv}{\tilde{v}}
	\newcommand{\betaa}{\pmb{\beta}}
	\newcommand{\deltaa}{\pmb{\delta}}
	\newcommand{\bepsilon}{\bar{\epsilon}}

	\title[\resizebox{5.5in}{!}{Hecke algebras for tame genuine principal series and local shimura correspondence}]{Hecke algebras for tame genuine principal series and local shimura correspondence}

	\author{Runze Wang}
	\address{School of Mathematical Sciences, Zhejiang University, 866 Yuhangtang Road, Hangzhou, China 310058}
	\email{wang\underline{\ }runze@zju.edu.cn }

	\date{}
	\subjclass[]{}
	\maketitle
	
	\begin{abstract} 
		In this paper, we extend the local Shimura correspondence building upon the groundwork laid by Gordan Savin. As preparations, we review part of the type theory of Bushnell and Kutzko which equally applies to covering groups. By adapting the method of linear algebraic groups, we describe the structure of Hecke algebras associated with genuine principal series components and construct the types. In particular, we show each of them shares the same affine Hecke algebra part as one corresponding Hecke algebra of the principal endoscopy group. However,  we give a counter example to show they are not isomorphic in general.
	\end{abstract}
	
	\tableofcontents

	\section{Introduction} \label{S:intro}
	
	Let $ p $ be a prime number and $ F $ be a finite extension of $ \mathbbm{Q}_p $ with discrete valuation $ |\cdot|_F $. Let $ O_F \subset F$ be the ring of algebraic integers and $ \mathfrak{p}_F $ be its prime ideal. Let $ \kappa_{F}=O_F/\mathfrak{p}_F $ be the residue field with cardinality $ q $. Let $ \varpi \in O_F$ be a fixed uniformizer with $ |\varpi|_F=q^{-1} $.
	
	Let $ G= \mathbf{G} (F) $ be the group of $ F $-rational points of a connected split reductive group $  \mathbf{G}  $ over $ F $. Let $ \mathbf{T} \subset \mathbf{G} $ be a fixed maximal torus and $ \mathbf{B} $ be a Borel subgroup containing $ \mathbf{T} $. Write $ T $ and $ B $ for the corresponding groups of $ F $-rational points. Denote by $ \mathfrak{R}(G) $ the category of smooth complex representations of $ G $. Consider the Bernstein decomposition $$ \mathfrak{R}(G)=\prod_{\mathfrak{s} \in \mathfrak{B}(G)} \mathfrak{R}_{\mathfrak{s}}(G) ,$$where $ \mathfrak{B}(G) $ is an indexing set consisting of irreducible supercuspidal representations of Levi subgroups of $ G $ up to equivalence. 
	
	In \cite{Ro}, Roche constructs a type $ (J_\chi, \rho_\chi) $ for $ G $ in the sense of Bushnell and Kutzko \cite{BK} for a principal series component of $ \mathfrak{R}(G) $ determined by the character $ \chi: \mathbf{T}(O_F) \rightarrow \C^{\times} $. Precisely speaking, this means that a smooth irreducible representation $ (\pi, V_\pi) $, when restricted to $ J_\chi $, contains $ \rho_\chi $ if and only if $ \pi $ is equivalent to a subquotient of the normalized parabolic induction $ \text{Ind}_{B}^{G}(\widetilde{\chi}) $ where $ \widetilde{\chi} $ is a character of $ T $ extending $ \chi $. Denote by $ \mathcal{H}(G,\rho_\chi)  $ the convolution algebra of compactly supported, $ \rho_{\chi}^{-1} $-spherical functions on $ G $. When $ \chi $ is a depth zero character, i.e., $ \chi $ factors through $ \mathbf{T}(O_F) \rightarrow \mathbf{T}(\kappa_{F}) $, the structure of $ \mathcal{H}(G,\rho_\chi) $ is described in \cite{Go,Mo}. Roche extends the results to depth positive cases with mild restrictions on $ p $.
	
	Now we shift our focus on the covering groups. Assume that $ F $ contains the full group $ \mu_n $ of $ n $-th roots of unity. Consider a Brylinski--Deligne central extension
	$$\begin{tikzcd}
		\mu_n \ar[r, hook] & \overline{G} \ar[r, two heads] & G
	\end{tikzcd}$$
	of $ G $ by $ \mu_n $ which is incarnated by the pair $ (D, \eta= \mathbb{1}) $, see \cite{BD, GG, We6} for details. We are interested in the $ \epsilon $-genuine representations of $ \ol{G} $, where $ \mu_n $ always acts via a fixed embedding $ \epsilon: \mu_n \rightarrow \C^{\times} $.
	
	It is natural to explore the type theory and the corresponding Hecke algebras for the genuine Bernstein components of $ \ol{G} $. Savin initiated a study on the genuine unramified principal series component in \cite{Sav88}. Assume that $ p \nmid n $ and thus fix a splitting $ s $ of $ \ol{G} $ over $ \mathbf{G}(O_F) $. This gives a splitting of $ \ol{G} $ over the standard Iwahori subgroup $ I \subset \mathbf{G}(O_F)$. Let $ \rho: \ol{I} \rightarrow \C^{\times} $ be the genuine character of $ \ol{I} $ such that $ \rho(\zeta i)=\epsilon(\zeta) $ for all $ \zeta \in \mu_n $ and $ i \in s(I) $. There is an isomorphism known as ``local Shimura correspondence" between the genuine Iwahori Hecke algebra $ \mathcal{H}(\ol{G},\rho)$ and the Iwahori Hecke algebra $ \mathcal{H}(G_{Q,n},\mathbb{1}_{I_{Q,n}}) $, where $ G_{Q,n} $ is a linear algebraic group with the same dual group as $ \ol{G} $, see \cite{Sav04,GG}. Using a pivotal theorem  attributed to \cite[Theorem 8.3]{BK}, it is immediate that $ (\ol{I},\rho) $ is a type for the genuine unramified principal series component of $ \ol{G} $.
	
	As suggested in \cite[\S17]{GG}, there should perhaps be analogous results for other genuine Bernstein components of $ \ol{G} $ as well. This is the motivating point of this paper.
	
	We briefly outline the structure of this article below. 
	
	In Section 2, we introduce the basic set-up about Brylinski--Deligne covering groups. We recall some properties of the Weyl group and root system. We also review the structure of genuine pro-$ p $ Iwahori Hecke alegbra studied in \cite{GGK}.
		
	In Section 3, we review the Bushnell--Kutzko type theory, which is applicable to covering groups. We discuss the construction of types for covering torus.
	
	The fourth section and the fifth section contain the main results of this paper.
	\begin{thm} \label{Main THM}
		$ 	\text{\rm (Theorem \ref{type}, Theorem \ref{LSC})} $. Apply the restrictions to $ p $ as outlined in  $\text{\rm \S 4.2}$ and assume $ p \nmid n $. Fix a splitting $ s $ of $ \ol{G} $ over $ \mathbf{G}(O_F) $. Let $ \chi $ be a genuine character of $ \ol{\mathbf{T}}(O_F) $ and $ \pi_{\chi} $ be a genuine representation of $ \ol{T} $ such that $ \pi_{\chi}|_{\ol{\mathbf{T}}(O_F)} $ contains $ \chi $. Let $ T_{Q,n} $ be the maximal torus of $ G_{Q,n} $ and $ \mathbf{T}_{Q,n}(O_F) $ be the maximal compact subgroup in $T_{Q,n} $. Let $  \chi_{Q,n}: \mathbf{T}_{Q,n}(O_F) \rightarrow \C^{\times} $ be given by $ \chi_{Q,n}(y(a))=\chi(s(y(a))) $ for $ y \in Y_{Q,n} $ and $ a \in O_{F}^{\times} $. Let $ \pi_{\chi_{Q.n}} $ be an extension of $ \chi_{Q.n} $ to $ T_{Q,n} $.
		
		(i) There is an explicitly constructed pair $ (\ol{J_{\chi}},\rho_{\chi}) $ that is a $ [\ol{T},\pi_\chi]_{\ol{G}} $-type of $ \ol{G} $.
		
		(ii) Let $ (J_{\chi_{Q,n}},\rho_{\chi_{Q,n}}) $ be the type constructed with respect to  $ [T_{Q,n}, \pi_{\chi_{Q,n}} ]_{G_{Q,n}}$. There is a $ \C $-algebra isomorphism $$  \Upsilon:\mathcal{H}(W_{\chi_{Q,n},\text{\rm ex}}) \rightarrow \mathcal{H}(W_{\chi,\text{\rm ex}}), $$ where $ \mathcal{H}(W_{\chi_{Q,n},\text{\rm ex}}) $ is the affine Hecke algebra part in $ \mathcal{H}(G_{Q,n},\rho_{\chi_{Q,n}}) $ and $ \mathcal{H}(W_{\chi,\text{\rm ex}}) $ is a subalgebra in $ \mathcal{H}(W_{\chi,\text{\rm ex}}) $.
		
	\end{thm}
	 In Section 4, we define the pair $ (\ol{J_{\chi}},\rho_{\chi}) $ for a genuine character $ \chi $ of $ \ol{\mbf{T}}(O_F) $ following \cite{Ro}. We describe the structure of $  \mathcal{H}(\ol{G},\rho_{\chi})$ in Theorem \ref{thm:main} and Theorem \ref{thm:induction}, first considering the case of depth zero and then handling the general case with mild restrictions on $ p $ by induction. We show that $ (\ol{J_{\chi}},\rho_{\chi}) $ is a $ [\ol{T},\pi_\chi]_{\ol{G}} $-type, where $ \pi_\chi$ is a genuine representation of $ \ol{T} $ such that $ \pi_{\chi}|_{\ol{\mathbf{T}}(O_F)} $ contains $ \chi $. 
	 
     In Section 5, we discuss the relation between $\mathcal{H}(\ol{G},\rho_{\chi})  $ and a Hecke algebra $ \mathcal{H}(G_{Q,n},\rho_{\chi_{Q,n}}) $ of the linear algebraic group $ G_{Q,n} $, which generalizes the relevant results in \cite{Sav04}. We show the algebra $ \mathcal{H}(G_{Q,n},\rho_{\chi_{Q,n}}) $ shares the same affine Hecke algebra part as $\mathcal{H}(\ol{G},\rho_{\chi})  $.   However, they are not isomorphic in general and we give a counter example to show that.
	~\\
	\paragraph{\textbf{Acknowledgement}} I would like to thank Fan Gao for his advice and suggestions. This work would not be possible without his careful guidance. Thanks are also due to Edmund Karasiewicz for helpful comments and informative discussions.

	\section{Preliminaries}
	\subsection{Brylinski--Deligne central extension}
	Let $ \mathbf{G} $ be a split connected linear reductive group over $ F $ with root datum $$ (X,\Phi,\Delta; \ Y, \Phi^\vee, \Delta^\vee) ,$$ where $ X $ is the character lattice and $ Y $ is the cocharacter lattice of a fixed maximal split torus $ \mathbf{T} \subset \mathbf{G} $. Here $ \Delta $ is a choice of simple roots from the set $ \Phi $ of roots. 
	
	Let $ \mathbf{B}=\mathbf{T}\mathbf{U} $ be the Borel subgroup associated with $ \Delta $. Let $ W=N(\mathbf{T})/\mathbf{T} $ be the Weyl group of $ (\mathbf{G},\mathbf{T}) $. We fix a Chevalley-Steinberg system $ \{e_\alpha:\mathbf{G}_a \rightarrow \mathbf{U}_{\alpha}\}_{\alpha \in \Phi} $ of pinnings for $ (\mathbf{G},\mathbf{T}) $, where $ \mathbf{U}_\alpha \subset \mathbf{G}$ is the root subgroup associated with $ \alpha $.
	
	Let $$ D:Y \times Y \longrightarrow \Z $$be a bilinear form such that $ Q(y):=D(y,y) $ is a Weyl-invariant quadratic form. Let $$ B_Q(y,z):=Q(y+z)-Q(y)-Q(z) $$ be the Weyl-invariant bilinear form associated with $ Q $. For $ \alpha \in \Phi $, one has 
	\begin{equation}
		B_{Q}(y,\alpha^{\vee})=\langle \alpha, y \rangle Q(\alpha^{\vee}) .
	\end{equation} Let $$ \eta: Y^{\text{sc}} \longrightarrow F^{\times} $$ be a homomorphism of the coroot lattice $ Y^{\text{sc}} \subset Y $ into $ F^{\times} $.
	
	Let $ G=\mathbf{G}(F) $ be the $ F $-rational points of $ \mathbf{G} $ and we use such notations for other algebraic groups. Assume that $ F^{\times} $ contains the full group $ \mu_n $ of $ n $-th roots of unity. Then the pair $ (D,\eta) $ gives a Brylinski--Deligne central extension 
	
	$$\begin{tikzcd}
		\mu_n \ar[r, hook] & \overline{G} \ar[r, two heads] & G,
	\end{tikzcd}$$
	of $ G $ by $ \mu_n $. In this paper, we assume $ \eta= \mathbb{1} $.
	
	Consider the sublattice $ Y_{Q,n} $ given by $$ Y_{Q,n}=\{y \in Y: B_Q(y,z) \in n\Z  \text{ for all } z \in Y\} $$ and $ X_{Q,n}:=\text{Hom}(Y_{Q,n},\Z) $. For every $ \alpha \in \Phi$ we set $$ \alpha_{Q,n}^{\vee}=n_{\alpha}\alpha^{\vee} \text{ and } \alpha_{Q,n}=n_{\alpha}^{-1}\alpha$$ where $ n_{\alpha}=n/\text{gcd}(n,Q(\alpha^{\vee})) $. Denote $$ \Phi_{Q,n}^{\vee}=\{\alpha_{Q,n}^{\vee}: \alpha \in \Phi^{\vee}\} \text{ and } \Phi_{Q,n}=\{\alpha_{Q,n}: \alpha \in \Phi\}; $$ similarly for $ \Delta_{Q,n}^{\vee} $ and $ \Delta_{Q,n} $.
	
	Define the complex group $ \ol{G}^{\vee} $ associated with the root datum $$ (Y_{Q,n},\Phi_{Q,n}^{\vee},\Delta_{Q,n}^{\vee}; \ X_{Q,n}, \Phi_{Q,n}, \Delta_{Q,n}) $$
	to be the dual group of $ \ol{G} $. Let $ \mathbf{G}_{Q,n} $ be the split connected linear algebraic over $ F $ whose dual group is isomorphic to $ \ol{G}^{\vee} $. It is usually regarded as the principal endoscopy group of $ \ol{G} $. Let $\mathbf{T}_{Q,n}   $ be a fixed maximal split torus of $  \mathbf{G}_{Q,n} $.
	
	The covering group $ \ol{G} $ splits over unipotent subgroups canonically and $ G $-equivariantly. Denote by $ \ol{e}_{\alpha}(F) $ the splitting of $ e_{\alpha}(F) $ in $ \ol{G} $. With this, define $$ \ol{w}_{\alpha}(x):= \ol{e}_{\alpha}(x)\ol{e}_{-\alpha}(-x^{-1})\ol{e}_{\alpha}(x)$$ and $$\ol{h}_{\alpha}(x):=\ol{w}_{\alpha}(x)\ol{w}_{\alpha}(-1) . $$
	
	There exists a section $ \mathbf{s} $ of $ \ol{T} $ over $ T $ such that for any $ y_1,y_2 \in Y $ and $ a,b\in F^{\times} $
	\begin{equation}  \label{E1}
		\mathbf{s}(y_1(a)) \cdot \mathbf{s}(y_2(b))=\mathbf{s}(y_1(a)\cdot y_2(b))\cdot(a,b)_n^{D(y_1,y_2)},
	\end{equation}
	where $ (-,-)_n $ is the $ n $-th Hilbert symbol. 
	
	By \eqref{E1}, the commutator $ [-,-]:T \times T \rightarrow \mu_n $ is given by
	\begin{equation} \label{E4}
		[y_1(a),y_2(b)]=(a,b)_n^{B_Q(y_1,y_2)}.
	\end{equation}
	
	In addition, we list some relations used in computation. Let $ \alpha, \beta \in \Phi $, $ y \in Y $, $ u \in F $, $ x,x_1,x_2 \in F^{\times} $. Then one has:
	\begin{equation}  \label{E3}
		\ol{w}_{\alpha}(1) \cdot \mathbf{s}(y(x)) \cdot \ol{w}_{\alpha}(1)^{-1}=\mathbf{s}(y(x)) \cdot \ol{h}_{\alpha}(x^{-\langle \alpha,y \rangle})
	\end{equation}
	\begin{equation}  \label{E8}
	\ol{w}_{-\alpha}(x)=	\ol{w}_{\alpha}(-x^{-1})
    \end{equation}
	\begin{equation} \label{E5}
		\ol{w}_{\alpha}(x_1)\ol{w}_{\alpha}(x_2)=(-x_1,-x_2)^{Q(\alpha^{\vee})}_n \cdot \ol{h}_{\alpha}(-x_1x^{-1}_2)
	\end{equation}
	\begin{equation}  \label{E6}
		\ol{w}_{\alpha}(1) \cdot \ol{e}_{\beta}(u) \cdot \ol{w}_{\alpha}(1)^{-1}=\ol{e}_{\beta-\langle \beta,\alpha^{\vee} \rangle \alpha}(c(\alpha,\beta)u) \text{ with }  c(\alpha,\beta) \in \{\pm 1\}  
	\end{equation}
	\begin{equation} \label{E7}
		\ol{h}_{\alpha}(x) \cdot \ol{e}_{\beta}(u) \cdot \ol{h}_{\alpha}(x)^{-1}=\ol{e}_{\beta}(x^{\langle \beta,\alpha^{\vee} \rangle} u) 
	\end{equation}
	
	\subsection{Affine roots and the length function}
	Let $ V=Y \otimes \R$. For each $ \alpha \in \Phi $ and $ k \in \Z $, there is an affine linear function $ \alpha+k $ defined on $ V $, given by $ x \mapsto \alpha(x)+k $. Let $ \Phi_{\text{af}}=\{\alpha+k: \alpha \in \Phi, k \in \Z\} $. For each $ a \in  \Phi_{\text{af}} $, let $ H_a $ be the hyperplane on which $ a $ vanishes. Let $ s_{\alpha+k} $ be the corresponding reflection given by $$ s_{\alpha+k}(x)=x-(\alpha+k)(x)\alpha^{\vee}.$$
	
	To each $ v \in V $, we associate the translation $ \mathsf{t}_{v} $ which sends $ v^{\prime} \in V$ to $ v^{\prime}+v $. Let $ \text{Aff}(V) $ be the semidirect product of $ GL(V) $ and the group of translations by elements of $ V $. For $ r \in \text{Aff}(V) $, we define its action on affine linear functions by $$ (rf)(v):=f(r^{-1}v) ,$$where $ f:V \rightarrow \R $ is an affine linear function and $ v \in V $. 
	
	An alcove is a connected component of $ V-\bigcup_{a \in  \Phi_{\text{af}}}H_a $. Let $ \Phi^{+} \subset \Phi$ be the set of positive roots with respect to $ \Delta $. Then the fundamental alcove $ A_0 $ is given by $$ A_0:=\{x \in V :0<\alpha(x)<1 \text{ for all } \alpha \in \Phi^{+}\}. $$ 
	
	The alcove $ A_0 $ determines an ordering on $ \Phi_{\text{af}} $ via $ a>0 $ if and only if $ a(x)>0 $ for all $ x \in A_0 $. Denote by $ \Phi_{\text{af}}^{+} $ the set of positive affine roots and $ \Phi_{\text{af}}^{-} $ the set of negative affine roots.
	Let $ \Phi_{\text{max}} $ be the set of highest roots in $ \Phi $. Then $$ \Delta_{\text{af}}=\Delta \cup \{-\beta+1 :\beta \in \Phi_{\text{max}}\} $$is the unique set of simple roots associated with $ A_0 $. Let $ W_{\text{af}} $ be the affine Weyl group generated by $ s_a $ for all $ a \in \Phi_{\text{af}} $. The set of reflections $$ S_{\text{af}}:=\{s_{\alpha}: \alpha \in \Delta_{\text{af}}\}  $$realizes $ W_{\text{af}} $ as a Coxeter group. We know from \cite{Hum} that the length function $ l $ of $ (W_{\text{af}},S_{\text{af}}) $ satisfies $$ l(w)=\text{ number of hyperplanes } H_{a} \text{ that separate } w(A_0)  \text{ from } A_0.  $$
	
	Let $ W_{\text{ex}}:=Y \rtimes W $ be the extended affine Weyl group associated with $ G $. We extend $ l $ to a function $ W_{\text{ex}} \rightarrow \mathbf{Z}_{\geq 0} $ using the geometric description of length function above.
	
	Denote by $ N(w)=\{a \in \Phi_{\text{af}}^{+}:wa \in \Phi_{\text{af}}^{-}\} $ for $ w \in W_{\text{ex}} $. It is shown in \cite[\S 1.6]{Mo} that $$ \text{Card}(N(w))=l(w) $$where $ \text{Card}(N(w)) $ is the number of the elements in $ N(w) $, and for $ w_1,w_2 \in W_{\text{ex}} $ 
	\begin{equation} \label{geo-length}
		l(w_1w_2)=l(w_1)+l(w_2)  \text{ iff }  N(w_2) \subset N(w_1w_2). 
	\end{equation}
	In particular for $ a \in \Delta_{\text{af}} $, the formula for $ l(s_aw) $ is given by
	\begin{equation} \label{length}
		l(s_aw)=\left\{
		\begin{aligned}
			l(w)+1, & \text{ if } w^{-1}a \in \Phi_{\text{af}}^{+},\\
			l(w)-1, & \text{ otherwise}.
		\end{aligned}
		\right.
	\end{equation}
	
	\subsection{Genuine pro-$ p $ Iwahori Hecke algebra}
	In the following sections, we always assume $ p \nmid n $ and fix a splitting $ s $ of $ \ol{G} $ over $ \mathbf{G}(O_F) $. If there is no ambiguity, we write $ H $ instead of $ s(H) $ for a subgroup $ H \subset \mathbf{G}(O_F) $ to simplify notation and we always  write $ \ol{H} $ for $ \mu_n \times H $.
	
	Let $ I \subset \mathbf{G}(O_F)$ be the Iwahori subgroup associated with $ \mathbf{B}(\kappa_F) $ and $ I_1 $ be the unique maximal pro-$ p $ normal subgroup in $ I $. We fix the Haar measure $ \delta_{\ol{G}} $ on $ \ol{G} $ such that $ \delta_{\ol{G}}(\mu_n \times I_1)=1 $
	
	Denote by $ \mathcal{H}_{\bar{\epsilon}}(\ol{G},I_1):=C^{\infty}_{\bar{\epsilon},\text{c}}(I_1\backslash \ol{G}/I_1) $ the genuine pro-$ p $ Iwahori-Hecke algebra consisting of $ \bar{\epsilon} $-genuine locally constant compactly supported functions $ f $ on $ \ol{G} $ which are $ I_1 $-biinvariant, i.e, $ f(\zeta \gamma_1 g \gamma_2)=\epsilon(\zeta)^{-1}f(g) $ for all $ \zeta \in \mu_n $ and $ \gamma_1,\gamma_2 \in I_1 $. The algebra multiplication on $ \mathcal{H}_{\bar{\epsilon}}(\ol{G},I_1) $ is given by the convolution $ f_1 * f_2 $.
	
	For $ g \in \ol{G} $, let $ T_g \in \mathcal{H}_{\bar{\epsilon}}(\ol{G},I_1)$ be the unique element such that $ \text{supp}(T_g)=\mu_nI_1gI_1 $ and $ T_g(g)=1 $. For $ \alpha \in \Phi $ and $ \chi \in \text{Hom}(\kappa_F^{\times},\mu_n) $ we define $$ c_\alpha(\chi):=\frac{1}{q-1}\sum_{u \in \kappa_F^{\times}}\epsilon(\chi(u))T_{\ol{h}_{\alpha}(u)}\in \mathcal{H}_{\bar{\epsilon}}(\ol{G},I_1).$$
	
	Let $ N(\ol{T}) $ be the normalizer of $ \ol{T} $ in $ \ol{G} $. Then $ \ol{w}_{\alpha}(1) \rightarrow s_{\alpha} $ and $ \mathbf{s}(y(\varpi^{-1})) \rightarrow y $ give rise to an isomorphism $$ \text{pr}:N(\ol{T})/\ol{\mathbf{T}}(O_F) \rightarrow   W_{\text{ex}}.$$We say $ \dot{w} \in N(\ol{T})$ is a representative of $ w \in W_{\text{ex}} $ if $ \text{pr}(\dot{w})=w $. We can lift the length function $ l $ to $ \ol{G} $ by setting $ l(g):=l(w) $ if $g \in \ol{I}\dot{w}\ol{I}  $.

	It is shown in \cite[\S3.2]{GGK} that the following braid relations and quadratic relations hold for $ \mathcal{H}_{\bar{\epsilon}}(\ol{G},I_1) $.
	
	\begin{prop} \label{lm:br}
		{\rm(Braid relations)}. Let $ g,g^{\prime} \in \ol{G} $. If $ l(g)+l(g^{\prime})=l(gg^{\prime}) $, then 
		\begin{equation} \label{quad1}
			T_g * T_{g^{\prime}}=T_{gg^{\prime}}.
		\end{equation}
	\end{prop}

	\begin{prop} \label{lm:qu}
		{\rm(Quadratic relations).} 
		\begin{itemize}
			\item[(i)] For every $ \alpha \in \Delta $ 
			\begin{equation} \label{quad1}
				T_{\ol{w}_{\alpha}(1)}^2=qT_{\ol{h}_{\alpha}(-1)}+(q-1)c_{\alpha}(\mathbf{1})T_{\ol{w}_{\alpha}(1)}.
			\end{equation}
			\item[(ii)] For every $ \beta \in \Phi_{\text{max}} $
			\begin{equation} \label{quad2}
				T_{\ol{w}_{\beta}(\varpi^{-1})}^2=q\epsilon((\varpi,\varpi)_{n}^{Q(\beta^\vee)}) T_{\ol{h}_{\beta}(-1)}+(q-1)c_{\beta}((-,\varpi)_n^{Q(\beta^\vee)})T_{\ol{w}_{\beta}(\varpi^{-1})}.
			\end{equation}
		\end{itemize}
	\end{prop}
	\begin{rmk}
		Here $ \epsilon((\varpi,\varpi)_{n}^{Q(\beta^\vee)}) T_{\ol{h}_{\beta}(-1)}= T_{\ol{w}_{\beta}^2(\varpi^{-1})}$ by \eqref{E5}.
		
	\end{rmk}

	\section{Review of Bushnell--Kutzko type theory}
	In this section, we discuss briefly the theory of Bushnell--Kutzko types. Althrough in \cite{BK} the type theory is studied in the setting of linear algebraic groups, it works almost word for word for the covering groups.
	
	\subsection{Bernstein decomposition for covering groups}
	Let $ \ol{G}^{\circ} \subset \ol{G}$ be the subgroup generated by all compact subgroups of $ \ol{G} $. The set of unramified characters of $ \ol{G} $ is defined to be
	$$ \mathcal{X}(\ol{G}):=\text{Hom}(\ol{G}/ \ol{G}^{\circ} ,\C^{\times}). $$We consider pairs $ (\ol{L},\pi) $ consisting of a Levi subgroup $ \ol{L}\subset \ol{G} $ and an irreducible supercuspidal representation $ \pi $ of $ \ol{L} $. Two such pairs $ (\ol{L}_j,\pi_j) $, $ j=1,2 $ are called inertially equivalent if there exist $ g \in \ol{G} $ and $ w \in \mathcal{X}(\ol{L}_2) $ such that 
	$$ \ol{L}_2=\ol{L}_1^g \text{ and } \pi_1^g \simeq \pi_2 \otimes w,$$where $ \ol{L}_1^g:=g^{-1}\ol{L}_1g $ and $ \pi_1^g(x):=\pi_1(gxg^{-1})  $ for every  $ x \in   \ol{L}_1^g$. Denote by $ [\ol{L},\pi]_{\ol{G}} $ the inertial equivalence class of the pair $ (\ol{L},\pi) $. Let
	$ \mathfrak{B}(\ol{G})$
	be the set of all inertial equivalence classes in $ \ol{G} $.
	
	Denote by $ \mathfrak{R}(\ol{G}) $ the category of smooth representation of $ \ol{G} $. Let $ \Pi \in \mathfrak{R}(\ol{G})$ be an irreducible representation. We define its support to be the pair $ (\ol{L},\pi) $ up to conjugation such that $ \Pi $ isomorphic to a subquotient in the normalized parabolic induction $ \text{Ind}_{\ol{P}}^{\ol{G}}(\pi) $, where $ \ol{P} $ is a parabolic subgroup of $ \ol{G} $ with Levi component $ \ol{L} $. The inertial equivalence class
	$$ \mathfrak{I}(\pi):=[\ol{L},\pi]_{\ol{G}} \in \mathfrak{B}(\ol{G}) $$is then called the inertial support of $ \Pi $. For any $ \mathfrak{s} \in \mathfrak{B}(\ol{G}) $, we have the subcategory
	$$ \mathfrak{R}_{\mathfrak{s}}(\ol{G}):=\{\Pi^{\prime}\in \mathfrak{R}(\ol{G}):\text{ every irreducible subquotient }\Pi \text{ of } \Pi^{\prime} \text{ satisfies } \mfr{I}(\Pi)=\mfr{s}\} $$of $ \mfr{R}(\ol{G}) $. As in the linear case, we have the Bernstein decomposition for covering groups 
	\begin{equation} \label{Bernstein}
		\mathfrak{R}(\ol{G})=\prod_{\mathfrak{s} \in \mathfrak{B}(\ol{G})} \mathfrak{R}_{\mathfrak{s}}(\ol{G}) ,
	\end{equation}
	see \cite[Theorem 7.14]{FP}.
	\subsection{Bushnell--Kutzko types for covering groups}
	Let $ H \subset \ol{G} $ be a open compact subgroup, with measure inherited from a Haar measure $ \delta_{\ol{G}} $ of $ \ol{G} $. Let $ (\rho,W) $ be a smooth irreducible representation of $ H $ with character denoted by $ \chi_{\rho} $. For $ (\pi,V) \in \mfr{R}(\ol{G})$, we denote by $ V^{\rho} $ the $ \rho $-isotypic subspace of $ V $.
	
	Let $ \mathcal{H}(\ol{G}):=C^{\infty}_{c}(\ol{G}) $ be the convolution algebra of locally constant compactly supported functions on $ \ol{G} $. We have $ V^{\rho}=e_{\rho}V $, where $ e_{\rho} \in  \mathcal{H}(\ol{G}) $ is the idempotent given by 
	$$ e_{\rho}(x)=\left\{\begin{array}{ll}
		\delta_{\ol{G}}(H)^{-1}\cdot\text{dim}(\rho)\cdot \chi_{\rho}(x^{-1}) & \text{ if }x \in H, \\
		0 & \text{ otherwise}.
	\end{array}\right. $$Consider the full subcategory
	$$ \mfr{R}_{\rho}(\ol{G})=\{(\pi,V)\in \mfr{R}(\ol{G}): V \text{ is generated by }V^{\rho}\} $$of $ \mfr{R}(\ol{G}) $. One has a functor 
	$$ \mfr{R}_{\rho}(\ol{G}) \longrightarrow \mca{M}(e_{\rho}\mca{H}(\ol{G})e_{\rho}) \text{ given by } V \mapsto V^{\rho} ,$$where $ \mca{M}(A) $ denotes the category of modules over an algebra $ A $. However, this functor does not give an equivalence of categories in general.
	
	\begin{dfn} \label{Def:cover}
		Let $ \mfr{S}\in \mfr{B}(\ol{G}) $ be a finite subset. A pair $ (H,\rho) $ is called an $ \mfr{S} $-type of $ \ol{G} $ if it satisfies the following property: an irreducible representation $ \Pi $ of $ \ol{G} $ contains $ \rho $ if and only if $ \mfr{I}(\Pi) \in \mfr{S}$.
	\end{dfn}
	
	If $ \mfr{S}=\{\mfr{s}\} $ is a singleton set, then we simply use the term $ \mfr{s} $-type. The following result is from \cite[Theorem 7.5]{Ro}. The essential ingredient used in the proof of loc. cit. is the Bernstein decomposition, which also holds for covering groups as in \eqref{Bernstein}.
	
	\begin{prop} \label{equivalence}
		Let $ \mfr{s}\in \mfr{B}(\ol{G}) $ and $ (H,\rho) $ be an $ \mfr{s} $-type. Then the following hold:
		\begin{itemize}
			\item[(i)] As subcategories of $ \mfr{R}(\ol{G}) $, one has
			$$ \mfr{R}_{\rho}(\ol{G})=\mfr{R}_{\mfr{s}}(\ol{G}). $$ 
			In particular, $ \mfr{R}_{\rho}(\ol{G}) $ is closed under taking subquotients.
			\item[(ii)] The functor 
			$$ \mfr{R}_{\rho}(\ol{G}) \longrightarrow \mca{M}(e_{\rho}\mca{H}(\ol{G})e_{\rho}),\quad V \mapsto V^{\rho} $$
			gives an equivalence of categories.
		\end{itemize}
	\end{prop}

	\subsection{Construction of types}
	We briefly review \cite[\S 6-8]{BK}, which gives a procedure for constructing types in $ \ol{G} $ from types in a proper Levi subgroup $ \ol{L} $.
	
	Let $ \ol{P}_{u}=\ol{L}N_{u} \subset \ol{G}$ be a parabolic subgroup with Levi component $ \ol{L} $, where $ N_{u} $ is the unipotent part arising from the canonical splitting. Let $ N_l $ be the opposite of $ N_u $ relative to $ \ol{L} $. Let $ J \subset \ol{G} $ be an open compact subgroup satisfying $$ J=(J \cap N_l)\cdot(J\cap \ol{L})\cdot (J \cap N_u).  $$
	
	An element $ \zeta \in Z(\ol{L}) $ is called strongly positive relative to $ J $ and $ \ol{P}_u $ if it satisfies the following:
	\begin{itemize}
		\item[(i)] $\zeta J_u \zeta^{-1}\subset J_u \text{ and } \zeta^{-1}J_l \zeta \subset J_l$;
		\item[(ii)] for any open compact subgroups $ H_1,H_2 \subset N_u $, there exists an $ m \in \mathbf{Z}_{\geq 0} $ such that $ \zeta^{m}H_1\zeta^{-m}\subset H_2 $;
		\item[(iii)] for any open compact subgroups $ K_1,K_2 \subset N_l $, there exists an $ m \in \mathbf{Z}_{\geq 0} $ such that $ \zeta^{-m}K_1\zeta^{m}\subset K_2 $;
	\end{itemize} 
	
	It is shown in \cite[Lemma 2.7]{GSS} that $ Z(\ol{L})=\ol{Z(L)} \cap Z(\ol{T}) $, where $ \ol{Z(L)} $ is the preimage of $ Z(L) $. It is easy to see that $ \ol{Z(L)}^n \subset Z(\ol{L}) $ and thus the existence of strongly positive elements is ensured by \cite[Proposition 6.14]{BK}.
	
	To proceed, let $ J_{\ol{L}} \subset \ol{L}$ be an open compact subgroup and $ (\rho_{\ol{L}},W) $ be an irreducible representation of $ J_{\ol{L}} $. Let $ (\rho,V) $ be an irreducible representation of $ J $ and $( \rho^{\vee},V^{\vee}) $ be the contragredient representation of $ \rho $. We denote by $ \mathcal{H}(\ol{G},\rho) $ the convolution algebra of compactly functions $$ f:\ol{G}\longrightarrow \text{End}_{\C}(V^{\vee}) $$satisfying $ f(agb)=\rho^{\vee}(a)f(g)\rho^{\vee}(b) $ for $ a,b \in J $ and $ g \in \ol{G} $.
	
	\begin{dfn} \label{G-cover}
		\cite[\S8]{BK} The pair $ (J,\rho) $ is called a $ \ol{G} $-cover of $ (J_{\ol{L}},\rho_{\ol{L}}) $ if the following hold:
		\begin{itemize}
			\item[(i)] For every parabolic subgroup $ \ol{P}=\ol{L}N_u \subset \ol{G} $ with Levi component $ \ol{L} $, one has $$ J=(J \cap N_l)\cdot(J\cap \ol{L})\cdot (J \cap N_u), $$
			and the groups $ J \cap N_l, J \cap N_u $ are both contained in the kernel of $ \rho $.
			\item[(ii)] $ J \cap \ol{L}=J_{\ol{L}} $ and $ \rho|_{J_{\ol{L}}} \simeq \rho_{\ol{L}}. $
			\item[(iii)] For every parabolic subgroup $ \ol{P} \subset \ol{G} $ with Levi component $ \ol{L} $, there exists a strongly positive element $ z_{\ol{P}}\in Z(\ol{L}) $ and an invertible element of $ \mathcal{H}(\ol{G},\rho) $ supported on the double coset $ Jz_{\ol{P}}J. $
		\end{itemize}
	\end{dfn}
	
	Let $ \mfr{S}_{\ol{L}}\subset \mfr{B}(\ol{L}) $ be a finite subset. Since each $ [\ol{M},\pi]_{\ol{L}} \in\mfr{B}(\ol{L}) $ determines an element $ [\ol{M},\pi]_{\ol{G}}\in \mfr{B}(\ol{G}) $, we denote by $ \mfr{S}_{\ol{G}} \subset \mfr{B}(\ol{G})$ the subset associated with $ \mfr{S}_{\ol{L}} $ in this way.
	
	\begin{thm} \label{thm:G-cover}
		Suppose that $ (J_{\ol{L}},\rho_{\ol{L}}) $ is an $ \mfr{S}_{\ol{L}} $-type of $ \ol{L} $ and $ (J,\rho) $ is a $ \ol{G} $-cover of $ (J_{\ol{L}},\rho_{\ol{L}}) $. Then $ (J,\rho) $ is an $ \mfr{S}_{\ol{G}} $-type of $ \ol{G} $.
	\end{thm}
	
	\begin{proof}
		We sketch the proof following closely \cite{BK}. For an irreducible representation $ (\Pi,V) $ of $ \ol{G} $, we consider its Jacquet module $ (\Pi_{N},V_N) $ with respect to the parabolic subgroups $ \ol{P}=\ol{L}N $. One has 
		\begin{equation} \label{Jacquet}
			V^{\rho} \simeq V_N^{\rho_{\ol{L}}}
		\end{equation}
		which follows from the same computation as in \cite[Theorem 7.9]{BK}. We need to show that $ V^{\rho} \neq 0$ if and only if $ \mfr{I}(\Pi) \in \mfr{S}_{\ol{G}} $.
		
		Suppose that $ V^{\rho} \neq 0 $. By the above isomorphism, we have $ V_N^{\rho_{\ol{L}}} \neq 0$. Since $ (J_{\ol{L}},\rho_{\ol{L}}) $ is an $ \mfr{S}_{\ol{L}} $-type, we see that $ \Pi_N $ has an irreducible subquotient $ \pi $ with $ \mfr{I}(\pi) \in \mfr{S}_{\ol{L}} $. By Frobenius reciprocity, $ \Pi $ embeds in the parabolic induction of $ \pi $ and thus $ \mfr{I}(\Pi) \in \mfr{S}_{\ol{G}}$.
		
		On the other hand, assume now that $ \mfr{I}(\Pi)=[\ol{M},\pi]_{\ol{G}} \in \mfr{S}_{\ol{G}}$ with $ [\ol{M},\pi]_{\ol{L}} \in \mfr{S}_{\ol{L}} $. By the definition of types,
		$$ \Pi \hookrightarrow \text{Ind}_{\ol{Q}}^{\ol{G}}(\pi \otimes \omega),$$where $ \ol{Q} \subset \ol{G} $ is a parabolic subgroup with Levi component $ \ol{M} $ and $ \omega \in \mca{X}(\ol{M}) $. Set 
		$$ \tau= \text{Ind}_{\ol{Q}\cap \ol{L}}^{\ol{L}}(\pi \otimes \omega)$$and let $ \ol{P}=\ol{L}N $ be a parabolic subgroup containing $ \ol{Q} $. By transitivity of parabolic induction, $ \Pi $ embeds in $ \text{Ind}_{\ol{P}}^{\ol{G}}(\tau) $. Then by the Frobenius reciprocity, $ \Pi_N $ and $ \tau $ have a composition factor in common. Because $ (J_{\ol{L}},\rho_{\ol{L}}) $ is an $ \mfr{S}_{\ol{L}} $-type, we have $ \Pi_N $ contains $ \rho_{\ol{L}} $ and it follows from \eqref{Jacquet} that $ \Pi  $ contains $ \rho $.
	\end{proof}
	
	\subsection{Types for covering torus}
	Consider the covering torus
	$$\begin{tikzcd}
		\mu_n \ar[r, hook] & \overline{T} \ar[r, two heads] & T,
	\end{tikzcd}$$which is a Heisenberg type group. There is a natural bijection \begin{equation} \label{rep cor}
		\{\text{genuine characters of }Z(\ol{T})\} \longleftrightarrow \{\text{irreducible genuine representations of }\ol{T}\},
	\end{equation}
	where $ Z(\ol{T}) $ is the center of $ \ol{T} $. 
	
	More precisely, let $ \chi:Z(\ol{T})\rightarrow \C^{\times} $ be a genuine central character. We can choose a maximal abelian group $ A \subset \ol{T} $ and an extension $ \widetilde{\chi}:A \rightarrow \C^{\times}  $ of $ \chi $. Then the induced representation $  \pi_{\chi}=\text{ind}_{A}^{\ol{T}}\widetilde{\chi} $ is irreducible and independent of the choice of $ A $ and $ \widetilde{\chi} $. It is the unique irreducible genuine representation of $ \ol{T} $ which has central character $ \chi $.
	
	Since we assume that $ p \nmid n$, we can take $ A=Z(\ol{T})\ol{\mathbf{T}}(O_F) $.
	
	\begin{lm} \label{lm:covering torus}
		The pair $ (\ol{\mathbf{T}}(O_F),\widetilde{\chi}|_{\ol{\mathbf{T}}(O_F)}) $ is a $ [\ol{T},\pi_{\chi}]_{\ol{T}} $-type of $ \ol{T} $.
	\end{lm}
	
	\begin{proof}
		We check this directly using the definition of types. On the one hand, for every $ \omega \in \mathcal{X}(\ol{T}) $ it is clear that $ (\pi_\chi \otimes \omega)|_{\ol{\mathbf{T}}(O_F)} $ contains $ \widetilde{\chi}|_{\ol{\mathbf{T}}(O_F)} $. On the other hand, if an irreducible representation $ \pi $ of $ \ol{T} $ contains $ \widetilde{\chi}|_{\ol{\mathbf{T}}(O_F)} $, then $ \pi|_{A} $ contains $  \widetilde{\chi}\otimes \omega_0 $, where $ \omega_0 $ is a character of $ A/\ol{\mathbf{T}}(O_F) $.
		Note that $ \omega_0 $ extends to a character $ \omega \in \mathcal{X}(\ol{T}) $. We have 
		$$ \pi \simeq \text{ind}_{A}^{\ol{T}} (\widetilde{\chi} \otimes \omega_0)=\pi_{\chi} \otimes \omega,$$by \eqref{rep cor}. This completes the proof.
	\end{proof}

	\section{The structure of Hecke algebras}
	In this section we study the structure of Hecke algebras associated with genuine principal series components of $ \ol{G} $.
	\subsection{Depth zero case}
	Let $ \chi: \ol{\mathbf{T}}(O_F) \rightarrow \C^{\times}$ be a depth zero genuine character of $ \ol{\mathbf{T}}(O_F) $, i.e., $ \chi|_{\mathbf{T}(O_F)} $ factors through $ \mathbf{T}(O_F) \rightarrow \mathbf{T}(\kappa_F)$. 
	
	The normalizer $ N(\ol{T}) \subset \ol{G}$ acts on characters of $ \ol{\mathbf{T}}(O_F) $ by $ (g\cdot\chi)(t)=\chi(g^{-1}tg) $ for $ g \in N(\ol{T}) $ and $ t \in \ol{\mathbf{T}}(O_F) $. Since $ Z(\ol{T}) \ol{\mathbf{T}}(O_F)  $ acts trivially, the action factors to an action of $ W_{\text{ex}}/Y_{Q,n} \simeq N(\ol{T})/Z(\ol{T})\ol{\mathbf{T}}(O_F) $ on the characters of $ \ol{\mathbf{T}}(O_F) $. 
	
	Let $ \ol{B}=\ol{T}U $ and $ U^{\prime} $ be the opposite of $ U $ relative to $ \ol{T} $. Consider the Iwahori decomposition $$ \ol{I}=(\ol{I}\cap U^{\prime})\cdot \ol{\mathbf{T}}(O_F) \cdot (\ol{I}\cap U)$$of $ \ol{I} $. We can extend $ \chi $ to a character $ \rho_\chi $ of $ \ol{I} $, i.e., $ \rho_\chi(a t b)=\chi(t) $ for $ a \in \ol{I}\cap U^{\prime} $, $ t \in \ol{\mathbf{T}}(O_F) $ and $  b \in \ol{I}\cap U $. Let $ \mathcal{H}(\ol{G},\rho_{\chi}) $ be the $ \rho_{\chi}^{-1} $-spherical Hecke algebra consisting of $ \bar{\epsilon} $-genuine locally constant compactly supported functions on $ \ol{G} $ such that $$ f( \gamma_1 g \gamma_2)=\rho_{\chi}(\gamma_1)^{-1}f(g)\rho_{\chi}(\gamma_2)^{-1} $$for all $ g \in \ol{G} $ and $ \gamma_1,\gamma_2 \in \ol{I} $. Let $ \mathcal{I}_{\chi} $ be the identity in $ \mathcal{H}(\ol{G}, \rho_\chi)$. Indeed $ \mathcal{H}(\ol{G},\rho_{\chi})  $ is a subalgebra in $ \mathcal{H}_{\bar{\epsilon}}(\ol{G},I_1) $.
	
	For $ \alpha \in \Phi $ and $ k \in \Z $, define $ \chi_{\alpha+k}:O_{F}^{\times}\rightarrow \C^{\times} $ by
	\begin{equation}
	\chi_{\alpha+k}(x):=\epsilon((\varpi,x)_n^{kQ(\alpha^{\vee})})\chi(\ol{h}_{\alpha}(x)).
	\end{equation} 
	\begin{prop} \label{lm:chi_aff}
		Let $ a=\alpha+k \in \Phi_{\text{{\rm af}}} $ and $ w \in W_{\text{{\rm ex}}} $. Then $ \chi_{a}=(w\cdot\chi)_{wa} $.
	\end{prop}
	
	\begin{proof}
		Write $ w=w_0y_0 $, where $ w_0 \in W $ and $ y_0 \in Y $. Then
		
		\begin{equation*}
			\begin{aligned}
				w(\alpha+k)&=
				(w_0y_0)(\alpha+k)\\
				&=w_0(\alpha+k-\langle \alpha,y_0\rangle)\\
				&=w_0\alpha+k-\langle \alpha,y_0\rangle.
			\end{aligned}
		\end{equation*}
		
		Thus for $ x \in O_F^{\times} $ we have 
		\begin{equation*}
			\begin{aligned}
				(w\cdot\chi)_{wa}(x)&=
				\epsilon((\varpi,x)_n^{(k-\langle \alpha,y_0\rangle)Q((w_0\alpha)^{\vee})})(w_0y_0 \cdot \chi)(\ol{h}_{w_0\alpha}(x))\\
				&=	\epsilon((\varpi,x)_n^{(k-\langle \alpha,y_0\rangle)Q((w_0\alpha)^{\vee})})(y_0 \cdot \chi)(\ol{h}_{\alpha}(x))\\
				&=\epsilon((\varpi,x)_n^{(k-\langle \alpha,y_0\rangle)Q((w_0\alpha)^{\vee})}) \chi((\varpi,x)_n^{\langle \alpha,y_0\rangle Q(\alpha^{\vee})}\ol{h}_{\alpha}(x)).
			\end{aligned}
		\end{equation*}
		
		Since $ Q $ is Weyl-invariant, we acquire $ \chi_{a}=(w\cdot\chi)_{wa} $.
	\end{proof}
	Let $$ \Phi_{\chi,\text{af}}=\{a\in \Phi_{\text{af}}:\chi_a=\mathbb{1}\}  \text{ and } W_{\chi}^0=\langle s_a :a \in \Phi_{\chi,\text{af}} \rangle.$$By Proposition \ref{lm:chi_aff}, reflections in $ W_{\chi}^0 $ preserve $ \Phi_{\chi,\text{af}} $. In fact, let $$ \Phi^{\diamondsuit}_{\chi}=\{\alpha \in \Phi:  \chi \circ (\ol{h}_{\alpha})^{n_{\alpha}}|_{O_F^{\times}}=\mathbb{1} \} \text{ and } \Phi_{\chi,\text{af}}^{\diamondsuit}=\{\alpha+kn_\alpha:\alpha \in \Phi^{\diamondsuit}_{\chi}, k \in \Z\}.$$We can show the affine Weyl groups associated with $ \Phi_{\chi,\text{af}} $ and $ \Phi_{\chi,\text{af}}^{\diamondsuit} $ are isomorphic.
	
	\begin{prop} \label{prop:shift}
		There exists $ v \in V $ such that $ \mathsf{t}_v(\Phi_{\chi,\text{\rm af}})= \Phi^{\diamondsuit}_{\chi,\text{\rm af}}$.
	\end{prop}
	
	\begin{proof}
		Since $ \epsilon((\varpi,-)^{Q(\alpha^{\vee})}_n) $ has order $ n_\alpha $, we observe that $ \alpha \in \Phi^{\diamondsuit}_{\chi} $ if and only if there exists a constant $ c_\alpha \in \Z $ such that $ \chi(\ol{h}_{\alpha}(x))=\epsilon((\varpi,x)_{n}^{-c_\alpha Q(\alpha^{\vee})}) $ for $ x \in O_F^{\times} $. Then $$ \Phi_{\chi,\text{af}}=\{\alpha+c_\alpha+kn_\alpha:\alpha \in \Phi^{\diamondsuit}_{\chi}, k \in \Z\}. $$ 
		
		Choose a set of roots $ \{\alpha_1,...,\alpha_m\} $ in $ \Phi^{\diamondsuit}_\chi $ such that $ \{\alpha_1^{\vee},...,\alpha_m^{\vee}\} $ is a set of simple roots in $ (\Phi^{\diamondsuit}_\chi)^{\vee} $. Without loss of generality, we assume $  \Phi^{\diamondsuit}_\chi  $ is irreducible. Then there are at most two root lengths and we denote by $ l^2 $ the ratio of square lengths between one long root and one short root. If $ l^2 \nmid n $, we can choose $ c_{\alpha_i} $ such that $ c_{\alpha_i}/l^2 \in \Z $.

		Take $ v \in V $ such that $ \langle \alpha_i,v \rangle=c_{\alpha_i} $ for $ i=1,...,m $. If $ \alpha^{\vee}=\sum_{i=1}^{m}k_i\alpha_i^{\vee} $ for $ \alpha \in \Phi^{\diamondsuit}_{\chi} $, then $$ n\mid \sum_{i=1}^{m}k_ic_{\alpha_i}Q(\alpha^{\vee}_i)-c_{\alpha}Q(\alpha^{\vee}). $$Let $ l^2(\alpha_i^{\vee},\alpha^{\vee}) $ be the ratio of square lengths between $ \alpha_i^{\vee} $ and $ \alpha^{\vee} $. Since $ Q $ is a Weyl-invariant quadratic form, we have $ \frac{Q(\alpha_i^{\vee})}{Q(\alpha^{\vee})}=l^2(\alpha_i^{\vee},\alpha^{\vee}) $. If $ l^2 \mid n $, then clearly $ \sum_{i=1}^{m}\frac{k_ic_{\alpha_i}Q(\alpha_i^{\vee})}{Q(\alpha^{\vee})}$ must be an integer and $ c_\alpha $ is of the form $ \sum_{i=1}^{m}\frac{k_ic_{\alpha_i}Q(\alpha_i^{\vee})}{Q(\alpha^{\vee})}+kn_\alpha $ for $ k \in \Z $. If $ l^2 \nmid n $, we have the same argument due to the choice of $ c_{\alpha_i} $.
		
		Since $ \alpha=\sum_{i=1}^{m}k_il^2(\alpha_i^{\vee},\alpha^{\vee})\alpha_i $, we have $$ \langle \alpha, v\rangle= \sum_{i=1}^{m}\frac{k_ic_{\alpha_i}Q(\alpha_i^{\vee})}{Q(\alpha^{\vee})}. $$Thus $ \mathsf{t}_{v}(\Phi_{\chi,\text{af}})= \Phi^{\diamondsuit}_{\chi,\text{af}}$.
	\end{proof}

	Let $ A_{\chi,0} $ be the connected component of $ V- \cup_{a \in  \Phi_{\chi,\text{af}}}H_a$ containing $ A_0 $. Let $ \Delta_{\chi} $ be the set of simple roots with respect to $ A_{\chi,0} $ and $ S_{\chi}^0=\{s_a:a \in \Delta_{\chi}\} $ be the set of corresponding reflections.
	
	Let $$W_{\chi}=\{w\in W_{\text{ex}}:w\cdot\chi =\chi\} \text{ and } \Omega_{\chi}=\{w\in W_{\chi} :wA_{\chi,0}=A_{\chi,0}\}. $$
	
	We note that for $ a=\alpha+k \in \Phi_{\chi,\text{af}} $ and $ x \in O^{\times}_F $ $$(s_a\chi)(\mathbf{s}(y(x)))=\chi(\mathbf{s}(y(x)))\chi_a(x^{-\langle \alpha,y \rangle})  $$by \eqref{E4} and \eqref{E3}, so $ W_{\chi}^{0} $ is a subgroup of $ W_{\chi}  $. Since $ W_{\chi} $ acts on the connected components of $ V- \cup_{a \in  \Phi_{\chi,\text{af}}}H_a$ and $ W_{\chi}^0 $ acts on them transitively, we have $ W_{\chi}=W_{\chi}^{0} \rtimes \Omega_{\chi}$.
	
	Denote $ \Lambda(w):=q^{l(w)}$ for $ w \in W_{\text{ex}} $. Define $ C_{w_1,w_2}:=\Lambda(w_1)^{\frac{1}{2}}\Lambda(w_2)^{\frac{1}{2}}\Lambda(w_1w_2)^{-\frac{1}{2}} $ for $ w_1,w_2 \in W_{\text{ex}} $. 
	
	The following three lemmas are adaptations of the results in \cite{Mo}.
	
	\begin{lm} \label{lm1}
		Let $ w \in W_{\text{\rm ex}} $ and $ a \in \Delta_{\text{\rm af}} $. Suppose either $ w^{-1}a>0 $, or $ w^{-1}a \notin \Phi_{\chi,\text{\rm af}} $. Let $ \dot{v} \in N(\ol{T}) $ be a representaive of $ s_a $ and $ \dot{w} \in N(\ol{T}) $ be a representative of $ w $. Let $ \chi^{\prime}=(s_aw)\cdot \chi $ and $ \mathcal{I}_{\chi^{\prime}}$ be the identity in $ \mathcal{H}(\ol{G}, \rho_{\chi^{\prime}}) $. Then 
		$$ \mathcal{I}_{ \chi^{\prime}} T_{\dot{v}}T_{\dot{w}}\mathcal{I}_{\chi}=C_{s_a,w} \mathcal{I}_{ \chi^{\prime}} T_{\dot{v}\dot{w}}\mathcal{I}_{\chi}.$$
	\end{lm}
	
	\begin{proof}
		If $ w^{-1}a>0 $, then by \eqref{length} we have $ l(s_aw)=l(s_a)+l(w) $ and $ \Lambda(s_aw)=\Lambda(s_a)\Lambda(w) $. Thus $ \mathcal{I}_{ \chi^{\prime}} T_{\dot{v}}T_{\dot{w}}\mathcal{I}_{\chi}=C_{s_a,w} \mathcal{I}_{ \chi^{\prime}} T_{\dot{v}\dot{w}}\mathcal{I}_{\chi} $ follows directly from Proposition \ref{lm:br}.
		
		Now suppose $ w^{-1}a<0 $, we can write $ w=s_aw^{\prime} $ such that $ l(w)=l(s_a)+l(w^{\prime}) $ and we take the representative $ \dot{w}^{\prime} $ such that $ \dot{w}=\dot{v}\cdot\dot{w}^{\prime} $. If $ a \in \Delta $, by \eqref{quad1}
		$$\mathcal{I}_{ \chi^{\prime}} T_{\dot{v}}T_{\dot{w}}\mathcal{I}_{\chi}=\mathcal{I}_{ \chi^{\prime}} T_{\dot{v}}T_{\dot{v}}T_{\dot{w}^{\prime}}\mathcal{I}_{\chi}= \mathcal{I}_{ \chi^{\prime}}(q T_{\dot{v}^2}+(q-1)c_a(\mathbf{1})T_{\dot{v}t_1})T_{\dot{w}^{\prime}}\mathcal{I}_{\chi}, $$where $ t_1$ is some element in $ \ol{\mathbf{T}}(O_F) $.
		But $ \mathcal{I}_{\chi^{\prime}}c_a(\mathbf{1})\neq 0 $ if and only if $( \chi^{\prime} )^{-1}_a=\mathbb{1}$. Since $ \chi^{\prime}=(s_aw)\cdot \chi $, then $ ( \chi^{\prime} )^{-1}_a=\chi_{w^{-1}a} $ by Proposition \ref{lm:chi_aff}. Thus if $ w^{-1}a \notin \Phi_{\chi,\text{af}} $, we have $  \mathcal{I}_{ \chi^{\prime}} T_{\dot{v}}T_{\dot{w}}\mathcal{I}_{\chi}=q \mathcal{I}_{ \chi^{\prime}} T_{\dot{v}\dot{w}}\mathcal{I}_{\chi}. $
		
		If $ a=-\beta+1 $ for some $ \beta \in \Phi_{\text{max}} $. Then by \eqref{quad2},
		$$ \mathcal{I}_{ \chi^{\prime}} T_{\dot{v}}T_{\dot{w}}\mathcal{I}_{\chi}=\mathcal{I}_{ \chi^{\prime}} T_{\dot{v}}T_{\dot{v}}T_{\dot{w}^{\prime}}\mathcal{I}_{\chi}= \mathcal{I}_{ \chi^{\prime}}(q T_{\dot{v}^2}+(q-1)c_{\beta}((-,\varpi)_n^{Q(\beta^{\vee})})T_{\dot{v}t_2})T_{\dot{w}^{\prime}}\mathcal{I}_{\chi}, $$where $ t_2$ is some element in $ \ol{\mathbf{T}}(O_F) $. Similarly, we see $ \mathcal{I}_{\chi^{\prime}}c_{\beta}((-,\varpi)_n^{Q(\beta^{\vee})})\neq 0 $ if and only if $ (\chi^{\prime})^{-1}_{\beta}=\epsilon((-,\varpi)_n^{Q(\beta^{\vee})})$. This is equivalent to $ \chi^{\prime}_{-\beta+1}=\mathbb{1} $. By the same argument, we conclude that  $ \mathcal{I}_{\chi^{\prime}}c_{\beta}((-,\varpi)_n^{Q(\beta^{\vee})})= 0 $ if $ w^{-1}a \notin \Phi_{\chi,\text{\rm af}} $.
		
		By the definition of $ C_{s_a,w}$, it is easy to obtain $ C_{s_a,w}=q $ if $ w^{-1}a <0 $,
		which completes the proof.
	\end{proof}
	
	\begin{lm} \label{lm2}
		Let $ w \in W_{\chi} $ and $ v \in W_{\text{\rm ex}} $. Let $ \dot{w} $ and $ \dot{v} $ be corresponding representatives in $ N(\ol{T}) $. Suppose that $ N(w^{-1})\cap N(v) \cap \Phi_{\chi,\text{\rm af}}=\emptyset $. Let $ \chi^{\prime}=(vw)\cdot \chi $ and $ \mathcal{I}_{\chi^{\prime}}$ be the identity in $ \mathcal{H}(\ol{G}, \rho_{\chi^{\prime}}) $. Then $$\mathcal{I}_{\chi^{\prime}} T_{\dot{v}}T_{\dot{w}}\mathcal{I}_{\chi}=C_{v,w} \mathcal{I}_{\chi^{\prime}} T_{\dot{v}\dot{w}}\mathcal{I}_{\chi} .$$ 
	\end{lm} 
	
	\begin{proof}
		Use induction on $ l(v) $. If $ l(v)=0 $ then $ l(vw)=l(v)+l(w) $ and the result follows from the the braid relations of pro-$ p $ Iwahori Hecke algebra $ \mathcal{H} $.
		
		Suppose $ l(v)>0 $, we can find $ a \in \Delta_{\text{af}} $ such that $ v=s_au $ and $ l(v)=l(s_a)+l(u) $. So by \eqref{geo-length}, $ N(u) \subseteq N(v) $. Let $ \dot{v}_1 $ be a representative for $ s_a $ and $ \dot{u} $ be a representative for $ u $ such that $ \dot{v}=\dot{v}_1 \dot{u} $. We can use induction to get
		$$\mathcal{I}_{(uw)\cdot \chi} T_{\dot{u}}T_{\dot{w}}\mathcal{I}_{\chi}=C_{u,w} \mathcal{I}_{(uw)\cdot \chi} T_{\dot{u}\dot{w}}\mathcal{I}_{\chi} .$$
		
		Now we need to check either $ w^{-1}u^{-1}a>0 $ or $ w^{-1}u^{-1}a \notin \Phi_{\chi,\text{af}} $. Suppose $ w^{-1}u^{-1}a<0 $, then $ u^{-1}a \notin N(v) \cap \Phi_{\chi,\text{af}} $. But $ v(u^{-1}a)=-a<0 $. Then $ u^{-1}a \notin \Phi_{\chi,\text{af}} $ and $ w^{-1}u^{-1}a \notin \Phi_{\chi,\text{af}} $ since $ w $ preserves $ \Phi_{\chi,\text{af}} $.
		
		We can verify directly that $ C_{v,w}=C_{s_a,uw}\cdot C_{u,w}$. Then using Lemma \ref{lm1}, we obtain the lemma.
	\end{proof}
	
	Define $ \ol{w}_{a}:=\ol{w}_{\alpha}(\varpi^{k}) $ for $ a=\alpha+k \in \Phi_{\text{af}}$.
	\begin{lm} \label{lm3}
		Let $ a=\alpha+k \in \Delta_{\chi} $ and $ v=s_a $.  Let $ E_{ \ol{w}_{a}}:=\Lambda(v)^{-\frac{1}{2}}\mathcal{I}_{\chi}T_{ \ol{w}_{a}}\mathcal{I}_{\chi} $. Then
		$$E_{ \ol{w}_{a}}^2=\mathcal{I}_{\chi}+q^{-\frac{1}{2}}(q-1)E_{ \ol{w}_{a}}. $$
	\end{lm}
	
	\begin{proof}
		As shown in \cite[\S2]{Mo}, there exists an element $ w \in W_{\text{af}} $ and $ b \in \Delta_{\text{af}} $ such that $ w^{-1}s_bw=v $ and $ l(v)=l(s_b)+2l(w) $.
		
		Since $ A_{\chi,0} $ contains $ A_{0} $ by definition, we observe that $ a^{\prime}(x)>0 $ for all $x\in A_{\chi,0}  $ if and only if $ a^{\prime}(x)>0 $ for all $ x\in A_0 $, where  $ a^{\prime}$ is an affine root in  $\Phi_{\chi,\text{af}} $. Then we see $ N(v) \cap \Phi_{\chi,\text{af}}=\{a\} $ since $ a \in \Delta_{\chi} $. But $ N(w) \subset N(v) $ by \eqref{geo-length} and $ wa=b \in \Delta_{\text{af}} $. Thus the set $ N(w) \cap \Phi_{\chi,\text{af}} $ is empty. 
		
		Assume $ w=s_{a_1}...s_{a_k} $ is a reduced expression for $ w $. We take $ \dot{w}=\ol{w}_{a_1}...\ol{w}_{a_k}$. Then $ \dot{w}^{-1}\ol{w}_b \dot{w}=\ol{w}_{\alpha}(\varpi^{k})$ or $ \ol{w}_{\alpha}(-\varpi^{k}) $ by \eqref{E6} and \eqref{E7}. However, by computation $$ \chi(\ol{w}_{\alpha}(\varpi^{k})\ol{w}^{-1}_{\alpha}(-\varpi^{k}))=\chi((\varpi,-1)^{kQ(\alpha^{\vee})}_n\ol{h}_{\alpha}(-1))=1 .$$
		Thus $\mathcal{I}_{\chi}T_{\ol{w}_{\alpha}(\varpi^{k})}\mathcal{I}_{\chi}= \mathcal{I}_{\chi}T_{\ol{w}_{\alpha}(-\varpi^{k})}\mathcal{I}_{\chi} $.
		Then by Proposition \ref{lm:br} and Lemma \ref{lm2}, we have	$$ \mathcal{I}_{w \chi}T_{\dot{w}}T_{\ol{w}_a}\mathcal{I}_{\chi}=C_{w,v}\mathcal{I}_{w \chi}T_{\ol{w}_b}T_{\dot{w}}\mathcal{I}_{\chi} $$
		and \begin{equation*}
			\begin{aligned}
				\mathcal{I}_{w \chi}T_{\dot{w}}T_{\ol{w}_a}T_{\ol{w}_a}\mathcal{I}_{\chi}&=
				C_{w,v}\mathcal{I}_{w \chi}T_{\ol{w}_b}T_{\dot{w}}T_{\ol{w}_a}\mathcal{I}_{\chi} \\
				&= C_{w,v}^2\mathcal{I}_{w \chi}T_{\ol{w}_b}T_{\ol{w}_b}T_{\dot{w}}\mathcal{I}_{\chi}.
			\end{aligned}
		\end{equation*}
		
		Since $ (w\chi)_b=\chi_{w^{-1}b}=\chi_a=\mathbb{1} $, by Proposition \ref{lm:qu} $$ 	\mathcal{I}_{w \chi} T_{\ol{w}_b}T_{\ol{w}_b}=q\mathcal{I}_{w \chi}+(q-1)\mathcal{I}_{w \chi}T_{\ol{w}_b}.$$
		
		Thus $ \mathcal{I}_{w \chi}T_{\dot{w}}T_{\ol{w}_a}T_{\ol{w}_a}\mathcal{I}_{\chi}=qC_{w,v}^2\mathcal{I}_{w \chi}T_{\dot{w}}\mathcal{I}_{\chi}+(q-1)	C_{w,v}\mathcal{I}_{w \chi}T_{\dot{w}}T_{\ol{w}_a}\mathcal{I}_{\chi} .$
		
		Indeed since $ \Lambda(v)=q\cdot\Lambda(w)^2 $, we can compute that $ C_{w,v}=q^{-\frac{1}{2}}\Lambda(v)^{\frac{1}{2}} $, which completes the proof.
	\end{proof}

	Now we can summarize the structure of $ \mathcal{H}(\ol{G},\rho_{\chi}) $.
	\begin{thm} \label{thm:main}
		Let $ \dot{w} \in N(\ol{T})$ be a representative of  $ w \in W_{\chi} $. Let \begin{equation} \label{E_w}
			E_{\dot{w}}:=\Lambda(w)^{-\frac{1}{2}}\mathcal{I}_{\chi}T_{\dot{w}}\mathcal{I}_{\chi} .
		\end{equation}
		
		\begin{itemize}
			\item[(i)] 	For each $ w \in W_{\chi} $, choose a representative $ \dot{w} \in N(\ol{T})$. Then  these $ E_{\dot{w}} $ form a basis of $ \mathcal{H}(\ol{G},\rho_{\chi}) $.
			\item[(ii)] For $ a \in \Delta_{\chi} $, 
			$E_{\ol{w}_a}^2=\mathcal{I}_{\chi}+q^{-\frac{1}{2}}(q-1)E_{\ol{w}_a} .$
			\item[(iii)]For $ a \in \Delta_{\chi} $ and $ w \in W_{\chi}^0 $, 	$E_{\ol{w}_a}E_{\dot{w}}=E_{\ol{w}_a\dot{w}}$ if $ w^{-1}a>0 $.
			\item[(iv)] For $ v \in W_\chi $ and $ w \in \Omega_{\chi} $, $ E_{\dot{v}}E_{\dot{w}}=E_{\dot{v}\dot{w}} $.
			
		\end{itemize}
	\end{thm}
	\begin{proof}
		The proof of (i) is the same as \cite[Proposition 4.1]{Ro}. The assertions (ii)-(iv) are direct consequences of Lemma \ref{lm1}, Lemma \ref{lm2} and Lemma \ref{lm3}.
	\end{proof}

	\subsection{General case} Now we consider the general case in which $ \chi $ may not be depth zero. The construction of the types and the technique of induction essentially follow from \cite{Ro}. 
	
	To use the arguments in loc. cit., we restrict $ p $ as follows if $ \Phi $ is irreducible :
	\begin{itemize}
		\item[$ \bullet $] for type $ A_m $\quad$p>m+1 $
		\item[$ \bullet $] for types $ B_m ,C_m,D_m$\quad$p\neq2 $
		\item[$ \bullet $] for type $ F_4$\quad$p\neq2,3 $
		\item[$ \bullet $] for types $ G_2,E_6$\quad$p\neq2,3,5 $
		\item[$ \bullet $] for types $ E_7,E_8$\quad$p\neq2,3,5,7 $
	\end{itemize}
	If $ \Phi $ is not irreducible, we exclude the ``bad" primes appearing in each of its irreducible factors. 
	
	For $ \alpha \in \Phi $, we define $ c_{\alpha} $ to be the least integer $ k \geq 1 $ such that $ 1+\mfr{p}_F^k \subset \text{Ker}(\chi \circ \ol{h}_{\alpha}) $. We define a function $ f_{\chi}:\Phi \rightarrow \mathbf{Z} $ as follows:
	$$
	f_\chi(\alpha)=\left\{\begin{array}{ll}
		\lfloor c_\alpha/2\rfloor & \text{ for } \alpha \in \Phi^{+}\\
		\lfloor (c_{\alpha}+1)/2\rfloor & \text{ for } \alpha \in \Phi^{-}
	\end{array}\right.
	$$where $ \lfloor x \rfloor $ denotes the largest integer $ \leq x $. 
	
	Let 
	\begin{equation} \label{eq:construction}
		U_{\chi}=\langle \ol{e}_{\alpha}(\mfr{p}^{f_\chi(\alpha)}):\alpha \in \Phi \rangle \text{ and } \ol{J_{\chi}}=\langle \mathbf{\ol{T}}(O_F), U_\chi \rangle.
	\end{equation}
	By \cite[\S3]{Ro}, $$\ol{J_{\chi}}/U_{\chi} \simeq \mathbf{\ol{T}}(O_F)/(U_\chi \cap \mathbf{\ol{T}}(O_F)) $$
	and $ (U_\chi \cap \mathbf{\ol{T}}(O_F))\subset \text{Ker}(\chi) $, we can lift $ \chi $ to a genuine character $ \rho_{\chi} $ of $ \ol{J_{\chi}} $.
	
	First we review a lemma in \cite[\S5]{Ro}.
	
	Let $ \ol{P}_{u}=\ol{L}N_{u} \subset \ol{G}$ be a parabolic subgroup with Levi component $ \ol{L} $, where $ N_{u} $ is the unipotent part arising from the canonical splitting. Let $ N_l $ be the opposite of $ N_u $ relative to $ \ol{L} $. Let $ (J,\tau) $ be a pair consisting of a compact open subgroup $ J $ of $ \ol{G} $ and a smooth character $ \tau $ of $ J $.

	\begin{lm} \label{iso-Hecke}
		Assume that $ (J,\tau) $ satisfies the following conditions:
		\begin{itemize}
			\item[(i)]  $ J=(J \cap N_l)\cdot(J\cap \ol{L})\cdot (J \cap N_u), $
			and the groups $ J \cap N_l, J \cap N_u $ are both contained in the kernel of $ \rho $.
			\item[(ii)] The support of $ \mathcal{H}(\ol{G},\tau) $ is contained in $ J\ol{L}J $.
		\end{itemize}
		
		Let $ J_{\ol{L}}= J \cap \ol{L}$ and $ \tau_{\ol{L}}=\tau|_{J_{\ol{L}}} $. Suppose $ \Phi \in \mathcal{H}(\ol{G}, \tau) $ is supported on $ JxJ $ for $ x \in \ol{L} $. Define $$ t(\Phi)=\frac{\delta_{\ol{G}}(JxJ)^{1/2}}{\delta_{\ol{L}}(J_{\ol{L}}xJ_{\ol{L}})^{1/2}}\phi ,$$where $ \delta_{\ol{L}} $ is a fixed Haar measure of $ \ol{L} $ and $ \phi $ is the unique element supported on $ J_{\ol{L}}xJ_{\ol{L}} $ such that $ \phi(x)=\Phi(x) $. Then $ t:\mathcal{H}(\ol{G},\tau)\rightarrow \mathcal{H}(\ol{L},\tau_{\ol{L}}) $ is a $ \C $-algebra isomorphism.
	\end{lm}
	
	\begin{proof}
		See \cite[\S 5]{Ro}. The proof could be carried out word for word.
	\end{proof}
	
	\begin{thm} \label{thm:induction}
		There is Levi subgroup $ \ol{L} $ of $ \ol{G} $ containing $ \ol{T} $ and a character of $ \chi_1 $ of $ \ol{L} $ such that $ \chi \otimes \chi_1 $ is a depth zero genuine character of $ \ol{\mathbf{T}}(O_F) $ and $$ \mathcal{H}(\ol{G},\rho_{\chi}) \simeq  \mathcal{H}(\ol{L},\rho_{\chi\otimes \chi_1}^{\ol{L}}), $$where $ \rho_{\chi\otimes \chi_1}^{\ol{L}} $ is constructed with respect to $ \ol{L} $ and the depth zero character $ \chi \otimes \chi_1 $.
	\end{thm}
	
	\begin{proof}
		First we claim that either the support of $ \mathcal{H}(\ol{G},\rho_{\chi}) $ has $ \ol{J_{\chi}}$-double coset representatives contained in a proper Levi subgroup of $ \ol{G} $ containting $ \ol{T} $ or there exists a character $ \chi_1 $ of $ \ol{G} $ such that $ \chi_1(\zeta)=1 $ for $ \zeta \in \mu_n $ and $ \chi \chi_1 $ has level zero.

		View $ \chi $ as a character of $ \mathbf{T}(O_F)\subset G $ by the restriciton $ \chi|_{s(\mathbf{T}(O_F))} $ and $ \rho_{\chi} $ as a character of $ J_{\chi} \subset G $ by the restriction $ \rho_{\chi}|_{s(J_{\chi})}  $. It is shown in \cite{Ro} that there is a pro-$ p $ subgroup $ K \subset J_\chi$ such that either the support of $ \mathcal{H}(G,\rho_{\chi}|_{K}) $ has $ K$-double coset representatives contained in a proper Levi subgroup of $ G $ containting $ T $ or there exists a character $ \chi_1 $ of $ G $ such that $ \chi \chi_1 $ has level zero, for details see \cite[\S4 and Theorem 6.3]{Ro}. 
		
		The key observation in the setting of covering group is that for $ g \in G $ and one of its preimage $ \ol{g} \in \ol{G}$, $$ s(K \cap gKg^{-1})=s(K) \cap \ol{g}s(K)\ol{g}^{-1} $$since the splitting of pro-$ p $ subgroup is unique. Since $ \ol{g}\in \ol{G} $ lies in the support of $ \mathcal{H}(\ol{G},\rho_{\chi}|_{\ol{K}}) $ if and only if $$ \rho_{\chi}|_{\ol{K} \cap \ol{g}\ol{K}\ol{g}^{-1}}=(\ol{g}\cdot\rho_{\chi})|_{\ol{K} \cap \ol{g}\ol{K}\ol{g}^{-1}}, $$we see the claim just follows from the situation of linear algebraic groups.
		
		The following proof is the same as that given in \cite[Theorem 6.3]{Ro}. We repeat it for completeness. 
		
		On the one hand, 	if there exists a character $ \chi_1 $ of $ \ol{G} $ such that $ \chi_1(\zeta)=1 $ for $ \zeta \in \mu_n $ and $ \chi \chi_1 $ has level zero. Then the map $ f \mapsto f\chi_1^{-1}:\mathcal{H}(\ol{G},\rho_\chi) \rightarrow \mathcal{H}(\ol{G},\rho_\chi\otimes \chi_1) $ is a support-preserving isomorphism. In fact using the construction in \cite[Theorem 4.15]{Ro}, we see $ J_{\chi}=J_{\chi\chi_1} $ and $ \rho_\chi\otimes \chi_1= \rho_{\chi \otimes \chi_1}$.
		
		On the other hand, if the support of $ \mathcal{H}(\ol{G},\rho_\chi) $ is contained in $ \ol{J_\chi}\bar{L} \ol{J_\chi} $ for a proper Levi subgroup $ \ol{L} $. 	It is shown in \cite[\S3]{Ro} that $ (\ol{J_\chi},\rho_\chi) $ satisfies the conditions in Lemma \ref{iso-Hecke}. Then there exists a support-preserving $ t:\mathcal{H}(\ol{G},\rho_{\chi}) \rightarrow \mathcal{H}(\ol{L},\rho_{\chi}^{\ol{L}}) $ by Lemma \ref{iso-Hecke}. Since $ \ol{J_\chi}\cap \ol{L} $ is the corresponding $ \ol{J_\chi} $ for $ \ol{L} $ and the semisimple rank of $ \ol{L} $ is strictly smaller than that of $ \ol{G} $, we can reduce to the case of semisimple rank zero. 
		
		But clearly there exists a character $ \chi_1 $ of $ \ol{T} $ such that $ \chi_1(\zeta)=1 $ for $ \zeta \in \mu_n $ and $ \chi\chi_1 $ has depth zero. Thus by induction we can construct the required isomorphism.
	\end{proof}
	Let $ \dot{w} \in N(\ol{T})$ be a representative of  $ w \in W_{\chi} $. Denote by $ E_{\dot{w}} \in  \mathcal{H}(\ol{G},\rho_{\chi}) $ the preimage of $  E_{\dot{w}}^{\ol{L}} $ defined in \eqref{E_w} for $  \mathcal{H}(\ol{L},\rho_{\chi\otimes \chi_1}^{\ol{L}}) $.
	
	\subsection{Types for genuine principal series representations}
	Let $ \pi_\chi$ be a genuine representation of $ \ol{T} $ such that $ \pi_{\chi}|_{\ol{\mathbf{T}}(O_F)} $ contains $ \chi $. 
	
	\begin{thm} \label{type}
		The pair $ (\ol{J_{\chi}},\rho_{\chi}) $ is a $ [\ol{T},\pi_\chi]_{\ol{G}} $-type of $ \ol{G} $.
	\end{thm}
	\begin{proof}
		We need to verify that that $ (\ol{J_{\chi}},\rho_{\chi}) $ is a $ \ol{G} $-cover of $ (\ol{\mathbf{T}}(O_F),\chi) $. The conditions (i) and (ii) in Definition \ref{Def:cover} can be easily checked, for details see \cite[Proposition 3.6]{Ro}.
		
		By Theorem \ref{thm:induction}, it suffices to check the condition (iii) in Definition \ref{Def:cover} for depth zero case. Let $ n \in \text{pr}^{-1}(W_{\chi}) $, then the non-zero element supported on $ InI $ can be written as a scale of $ E_{\ol{w}_{a_1}}\cdot\cdot\cdot E_{\ol{w}_{a_k}}E_{n^{\prime}} $ for $ a_1,...,a_k \in \Delta_{\chi} $ and $ n^{\prime} \in \text{pr}^{-1}(\Omega_{\chi})$. Since by Theorem \ref{thm:main} each $ E_{\ol{w}_{a_i}} $ is invertible and $ E_{n^{\prime}} $ is invertible, we see the non-zero element supported on $ InI $ is invertible.
		
		Thus by Theorem \ref{thm:G-cover} and Lemma \ref{lm:covering torus}, the pair $ (\ol{J_{\chi}},\rho_{\chi}) $ is a $ [\ol{T},\pi_\chi]_{\ol{G}} $-type of $ \ol{G} $.
	\end{proof}

	\section{Local Shimura correspondence}
	In this section, we discuss the relation between the Hecke algebra $ \mathcal{H}(\ol{G},\rho_{\chi}) $ and a Hecke algebra of the linear algebraic group $ G_{Q,n} $.
	
	\subsection{A character extension} Let $ W_{\chi}^{\diamondsuit}=\langle s_a:a \in \Phi^{\diamondsuit}_{\chi} \rangle$ and $W_{\chi,\text{af}}^{\diamondsuit}=\langle s_a:a \in \Phi^{\diamondsuit}_{\chi,\text{af}} \rangle $. By Proposition \ref{prop:shift}, there exists $ v \in V $ such that  $ \mathsf{t}_v(\Phi_{\chi,\text{af}})= \Phi^{\diamondsuit}_{\chi,\text{af}}$. Let $ A^{\diamondsuit}_{\chi,0} $ be the connected component of $ V- \cup_{a \in  \Phi^{\diamondsuit}_{\chi,\text{af}}}H_a$ containing $ A_0 $. Since $ W_{\chi,\text{af}}^{\diamondsuit} $ acts transitively on the components of $ V- \cup_{a \in  \Phi^{\diamondsuit}_{\chi,\text{af}}}H_a $, we can take $ w_0 \in W_{\chi,\text{af}}^{\diamondsuit} $ such that $ w_0\mathsf{t}_v(A_{\chi,0})=A^{\diamondsuit}_{\chi,0} $.
	
	\begin{eg}
		Let $ \Delta=\{\alpha,\beta\} $ be a fixed set of simple roots of $G= SL_3(\Q_p) $ with $ p \neq 2 $. Let $ \ol{G}=\ol{SL}^{(2)}_3(\Q_p) $ be the extension of $ SL_3(\Q_p) $ by quadratic form $ Q(\alpha^{\vee})=Q(\beta^{\vee})=1 $. 
		
		Let $ \chi:\ol{\mathbf{T}}(O_F)\rightarrow \C^{\times} $ be given by $ \chi\circ\ol{h}_\alpha|_{O_F^{\times}}=\chi\circ\ol{h}_\beta|_{O_F^{\times}}=\epsilon((\varpi,-)_2) $. Then as shown in Figure \ref{Figure1}, $ A_{\chi,0}$ is the triangle surrounded by thick yellow lines and $ A_{\chi,0}^{\diamondsuit}  $ is the triangle surrounded by thick green lines. We have $$\mathsf{t}_{-\alpha^{\vee}-\beta^{\vee}}(\Phi_{\chi,\text{af}})= \Phi^{\diamondsuit}_{\chi,\text{af}} \text{ and } s_{\alpha+\beta}\mathsf{t}_{-\alpha^{\vee}-\beta^{\vee}}(A_{\chi,0})=A^{\diamondsuit}_{\chi,0} .$$
		\begin{figure}[htbp]
			\includegraphics[scale=0.5]{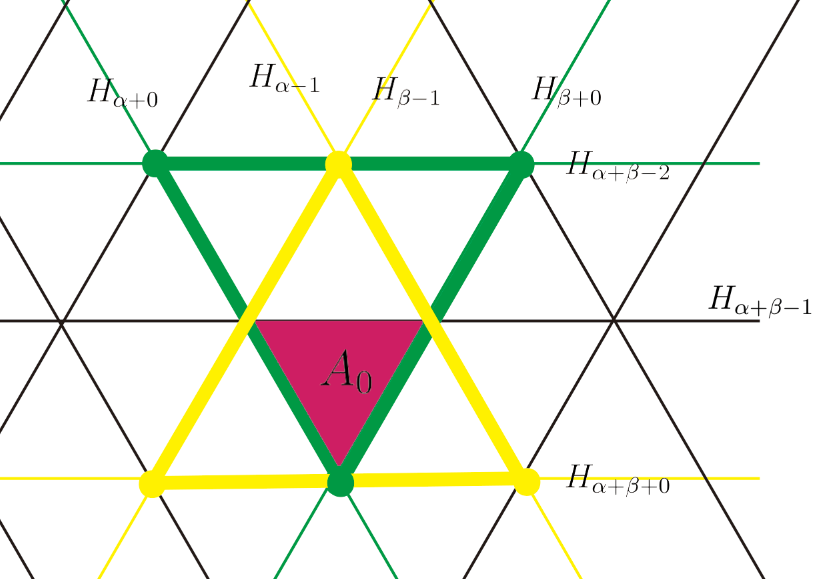}
			\caption{Affine apartment of $ SL_3$. The red part is the fundamental alcove $ A_0 $. The part surrounded by thick yellow lines is $ A_{\chi,0}$ and the part surrounded by thick green lines is $ A_{\chi,0}^{\diamondsuit}  $.}
			\label{Figure1}
		\end{figure}
	\end{eg}
	
	Let $ W_{\chi,\text{ex}}:=\mathsf{t}_v^{-1}(W_{\chi}^{\diamondsuit}\ltimes Y_{Q,n})\mathsf{t}_v=(\mathsf{t}_v^{-1}W_{\chi}^{\diamondsuit}\mathsf{t}_v)\ltimes Y_{Q,n}$. Since $$W^0_{\chi}=\mathsf{t}_v^{-1}W_{\chi,\text{af}}^{\diamondsuit}\mathsf{t}_v  \text{ and }\mathsf{t}_v^{-1}W_{\chi}^{\diamondsuit}\mathsf{t}_v \subset W^0_{\chi}  , $$we see $ W_{\chi,\text{ex}} $ is a subgroup of $ W_{\chi} $ containing $ W^{0}_{\chi} $.
	
	\begin{lm}
		The character $ \chi:\ol{\mathbf{T}}(O_F) \rightarrow \C^{\times}$ has an extension $$ \wt{\chi}:\text{\rm pr}^{-1}(W_{\chi,\text{\rm ex}}) \rightarrow \C^{\times} $$such that $  \wt{\chi}(\ol{w}_a)=1 $ for $ a \in \Delta_{\chi} $.
	\end{lm}
	
	\begin{proof}
		First, we show there exists a genuine character $ \wt{\chi}_1 $ of $ \text{pr}^{-1}(W_{\chi,\text{ex}}) $.
		
		Consider one covering group $ \ol{L}_{\chi} $ with root datum $ (X, \Phi_{\chi}^{\diamondsuit};Y,(\Phi_{\chi}^{\diamondsuit })^{\vee}) $ and the same Brylinski-Deligne pair $ (D, \eta)$ as $ \ol{G} $. Then $ \text{pr}^{-1}(W_{\chi,\text{ex}})  $ in $ \ol{L}_{\chi} $ is isomorphic to that in $ \ol{G} $ and clearly $ \text{pr}^{-1}((w_0\mathsf{t}_v)^{-1}W_{\chi}^{\diamondsuit}(w_0\mathsf{t}_v)) $ lies in a hyperspecial subgroup $ \ol{K} $ of  $ \ol{L}_{\chi} $.
		
		Let $ s^{\prime} $ be a splitting over $ K $. By \cite{GG}, there is a $ (w_0\mathsf{t}_v)^{-1}W_{\chi}^{\diamondsuit}(w_0\mathsf{t}_v) $-invariant genuine character $ \chi_1 :Z(\ol{T}) \rightarrow \C^{\times} $, which is trivial on the intersection of $ s^{\prime}(\mathbf{T}(O_F)) $. We can extend $ \chi_1 $ to get $ \wt{\chi}_1:\text{pr}^{-1}(W_{\chi,\text{ex}}) \rightarrow \C^{\times} $ such that $ \wt{\chi}_1(k)=1 $ for $ k \in \text{pr}^{-1}(W_{\chi,\text{ex}})\cap s^{\prime}(K) $.
		
		Then consider $ \chi_2=\chi\wt{\chi}^{-1}_1|_{\ol{\mathbf{T}}(O_F)} $, which is preserved by elements in $ W_{\chi,\text{ex}} $. Since $ \chi_2 $ is trivial on $ \mu_n $, this is essentially equivalent to the linear case. By \cite{HL} we can extend it to a character $ \wt{\chi}_2:\text{pr}^{-1}(W_{\chi,\text{ex}}) \rightarrow \C^{\times} $.
		
		Notice now $ \wt{\chi}_1\cdot\wt{\chi}_2 $ is an extension of $ \chi $ to $ \text{pr}^{-1}(W_{\chi,\text{ex}}) $. Then for $ a=\alpha+k \in \Delta_{\chi} $  $$( \wt{\chi}_1\cdot\wt{\chi}_2)(\ol{w}^2_a)=(\wt{\chi}_1\cdot\wt{\chi}_2)((-1,\varpi)^{kQ(\alpha^{\vee})}_n\ol{h}_{\alpha}(-1))=1 \text{ and }( \wt{\chi}_1\cdot\wt{\chi}_2)(\ol{w}_a)=\pm1.$$Now we define a character $ \chi_3:W_{\chi}\rightarrow \C^{\times} $ given by 	$$ \chi_3(w)=1 \text{ for }  w \in \Omega_{\chi} \text{ and } \chi_3(s_a)=( \wt{\chi}_1\cdot\wt{\chi}_2)(\ol{w}_{a})\text{ for }  a \in \Delta_{\chi}.$$We observe that if $ (s_{a_1}...s_{a_k})s_a (s_{a_1}...s_{a_k})^{-1}=s_b $, where $ a,b,a_1,...,a_k \in \Delta_{\chi} $, then by the argument in the proof of Lemma \ref{lm3} $$ (\wt{\chi}_1\cdot\wt{\chi}_2) ((\ol{w}_{a_1}...\ol{w}_{a_k})\ol{w}_a(\ol{w}_{a_1}...\ol{w}_{a_k})^{-1})=(\wt{\chi}_1\cdot\wt{\chi}_2)(\ol{w}_b).$$Thus $ \chi_3 $ is a well-defined character by the generating relations of $ W_{\chi} $.
		
		We lift $ \chi_3 $ to a character $ \wt{\chi}_3:\text{pr}^{-1}(W_{\chi})\rightarrow \C^{\times} $ and define $$ \wt{\chi}(n)=\wt{\chi}_1(n)\wt{\chi}_2(n)\wt{\chi}_3(n) \text{ for } n \in \text{pr}^{-1}(W_{\chi,\text{ex}}).$$Then the character $ \wt{\chi} $ of $ \text{pr}^{-1}(W_{\chi,\text{ex}}) $ is exactly an eligible extension.
	\end{proof}

	\subsection{Local Shimura correspondence}
	Let $$ \mathcal{H}(W_{\chi,\text{ex}}):=\{f \in \mathcal{H}(\ol{G},\rho_{\chi}):\text{supp}(f)\subset \ol{J_{\chi}}\text{pr}^{-1}(W_{\chi,ex})\ol{J_{\chi}}\}. $$For $ w \in W_{\chi,\text{ex}} $ and a representative $ \dot{w} \in N(\ol{T}) $ of $ w $, define \begin{equation} \label{def:ew}
		e_w:=\wt{\chi}^{-1}(\dot{w})E_{\dot{w}}
	\end{equation}
	which is indepedent of the choice of $ \dot{w} $. Using this basis, it is clear that $ \mathcal{H}(W_{\chi,\text{ex}}) $ is an affine Hecke algebra with parameter $ q $ by Theorem \ref{thm:main} and \ref{thm:induction}.

	Recall that the root datum of $ G_{Q,n} $ is given by $$ ( X_{Q,n}, \Phi_{Q,n}, \Delta_{Q,n};\ Y_{Q,n},\Phi_{Q,n}^{\vee},\Delta_{Q,n}^{\vee}).  $$Let $ T_{Q,n} $ be the maximal torus of $ G_{Q,n} $ with respect to $ Y_{Q,n} $. Let $ \mathbf{T}_{Q,n}(O_F) $ be the maximal compact subgroup in $T_{Q,n} $. There is a homomorphism $$ h: \mathbf{T}_{Q,n}(O_F) \rightarrow Z(\ol{T}) \cap s(\mathbf{T}(O_F)) $$given by $ h(y(a))=s(y(a)) $ for $ y \in Y_{Q,n} $ and $ a \in O_{F}^{\times} $. Then we can lift the character $ \chi:\ol{T}(O_F) \rightarrow \C^{\times} $ to a character $ \chi_{Q,n} $ of $ \mathbf{T}_{Q,n}(O_F) $. 
	
	Since linear algebraic groups are special cases of covering groups with $ n=1 $, we can transfer the previous notations for $ G_{Q,n} $ with respect to $ \chi_{Q,n} $.
	
	Let $ \Phi_{\chi_{Q,n}}=\{\alpha \in \Phi_{Q,n}:\chi_{Q,n}\circ \alpha^{\vee}|_{O_F^{\times}}=\mathbb{1}\} $. Then $ \Phi_{\chi_{Q,n},\text{af}}=\{\alpha+k:\alpha \in \Phi_{\chi_{Q,n}},k\in\Z\} $. 
	Let $ A_{\chi_{Q,n},0} $ be the connected component of $ V-\cup_{a \in \Phi_{\chi_{Q,n},\text{af}}}H_a $ determined by $ \Phi_{\chi_{Q,n}}\cap \Phi_{Q,n}^{+} $. It can be easily observed that $ \cup_{a \in \Phi_{\chi_{Q,n},\text{af}}}H_a= \cup_{a \in  \Phi^{\diamondsuit}_{\chi,\text{af}}}H_a $ and $  A_{\chi_{Q,n},0} =A_{\chi,0}^{\diamondsuit} $. Thus there is a natural isomorphism
	$$ \Psi_1: W_{\chi_{Q,n},\text{ex}} \rightarrow  W_{\chi}^{\diamondsuit}\ltimes Y_{Q,n}$$determined by $ \Psi_1(s_{\alpha/n_{\alpha}})=s_{\alpha} $ for $\alpha/n_{\alpha} \in  \Phi_{\chi_{Q,n}} $ and $ \Psi_1(y)=y $ for $ y \in Y_{Q,n} $. 
	
	Then we obtain an isomorphism $$ \Psi: W_{\chi_{Q,n},\text{ex}} \rightarrow W_{\chi,\text{ex}} $$given by $ \Psi(w)=(w_0\mathsf{t}_v)^{-1}(\Psi_1(w))(w_0\mathsf{t}_v) $.
	
	We write $ e^{Q,n}_w $ for the elements defined in \eqref{def:ew} with respect to $\mathcal{H}(W_{\chi_{Q,n},\text{ex}})  $.

	\begin{thm} \label{LSC}
		There is a $ \C $-algebra isomorphism $$  \Upsilon:\mathcal{H}(W_{\chi_{Q,n},\text{\rm ex}}) \rightarrow \mathcal{H}(W_{\chi,\text{\rm ex}}) $$given by $  \Upsilon(e^{Q,n}_w)=e_{\Psi(w)} $.
	\end{thm}
	
	\begin{proof}
		Since $ w_0\mathsf{t}_v(A_{\chi,0})=A_{\chi_{Q,n},0} $, the isomorphism $ \Psi $ induces a bijection between $S_{\chi_{Q,n}}^0 $ and $ S_{\chi}^0 $. At the same time, it induces an isomorphism between $ W_{\chi_{Q,n},\text{ex}} \cap \Omega_{\chi_{Q,n}} $ and $ W_{\chi,\text{ex}}\cap \Omega_{\chi} $. Since $$W_{\chi_{Q,n},\text{ex}}=W^0_{\chi_{Q,n}}\rtimes ( W_{\chi_{Q,n},\text{ex}} \cap \Omega_{\chi_{Q,n}}) \text{ and }W_{\chi,\text{ex}}=W^0_{\chi}\rtimes(W_{\chi,\text{ex}}\cap \Omega_{\chi}) ,$$it is obvious that $ \Upsilon $ is an isomorphism by Theorem \ref{thm:main} and \ref{thm:induction}.
	\end{proof}
	
	\begin{rmk}
		If $ W_{\chi_{Q,n}}=W_{\chi_{Q,n},\text{ex}} $, this implies the stabilizer of $ \chi|_{Z(\ol{T})\cap \ol{\mathbf{T}}(O_F)} $ in $ W_{\text{ex}} $ is equal to $ W_{\chi}^{\diamondsuit} \ltimes Y $ and thus $ W_{\chi}=W_{\chi,\text{ex}} $. Then $ \mathcal{H}(\ol{G},\rho_{\chi})=\mathcal{H}(W_{\chi,\text{ex}}) $ is isomorphic to $ \mathcal{H}(G_{Q,n},\rho_{\chi_{Q,n}})=\mathcal{H}(W_{\chi_{Q,n},\text{ex}}) $ by Theorem \ref{LSC}.
		
	\end{rmk}
	\begin{eg}
	However, in general the structure of $\mathcal{H}(\ol{G},\rho_{\chi})$ may be different from $ \mathcal{H}(G_{Q,n},\rho_{\chi_{Q,n}})$. Consider $ G=GL_2 $ with cocharacter lattice $ Y=\Z e_1 \oplus \Z e_2 $ and the set of simple coroot $ \Delta^{\vee}=\{\alpha^{\vee}:=e_1-e_2\} $. Take 
	the Weyl-invariant bilinear form $ B_Q $ determined by $$ B_{Q}(e_1,e_1)=B_{Q}(e_2,e_2)=2, B_Q(e_1,e_2)=-2. $$Let $ \ol{GL}^{(4)}_2(\Q_5) $ be a 4-fold covering group of $ GL_2 $ associated with the bilinear form. Take the character $ \chi $ such that $ \chi \circ \ol{h}_{\alpha}|_{O^{\times}_F} $ has order 2. By computation, we obtain $$ W_{\chi}=Y_{Q,n} \cup (s_{\alpha}e_1Y_{Q,n}) \text{ and } W_{\chi_{Q,n}}=\langle s_{\alpha} \rangle \ltimes Y_{Q,n},$$where $ Y_{Q,n}=\Z(e_1-e_2)\oplus\Z(e_1+e_2) $. Then $ W_{\chi} $ is free and it is not isomorphic to $ W_{\chi_{Q,n}} $. In this case, $ \Omega_{\chi}=W_{\chi} $ and $ \Omega_{\chi_{Q,n}}=W_{\chi_{Q,n}} $. By Theorem \ref{thm:main}, it is clear that there is an element of order $ 2 $ in $ \mathcal{H}(G_{Q,n},\rho_{\chi_{Q,n}}) $ but no such element in $\mathcal{H}(\ol{G},\rho_{\chi})$. Thus $ \mathcal{H}(\ol{G},\rho_{\chi}) $ and $ \mathcal{H}(G_{Q,n},\rho_{\chi_{Q,n}}) $ are not isomorphic.
	
	\end{eg}

\end{document}